\documentclass[12pt,reqno]{amsart}

\usepackage{amssymb,amsmath,amsbsy,eucal,amsfonts,mathrsfs,latexsym}
\usepackage{graphicx,verbatim,psfrag,epsfig,enumerate,color}
\usepackage{stmaryrd} 
\usepackage{lscape}
\usepackage{url}
\usepackage{subfig}
\usepackage{multirow}

\newtheorem{theorem}{Theorem}[section]
\newtheorem{lemma}{Lemma}[section]
\newtheorem{corollary}[theorem]{Corollary}

\theoremstyle{definition}

\newtheorem{remark}{{\it Remark}}[section]

\numberwithin{equation}{section}

\newcommand{\Bb}{{\boldsymbol{b}}}
\newcommand{\Bc}{{\boldsymbol{c}}}
\newcommand{\Bd}{{\boldsymbol{d}}}

\newcommand{\Bf}{{\boldsymbol{f}}}
\newcommand{\Bg}{{\boldsymbol{g}}}

\newcommand{\Bn}{{\boldsymbol{n}}}

\newcommand{\Bp}{{\boldsymbol{p}}}

\newcommand{\Bu}{{\boldsymbol{u}}}
\newcommand{\Bv}{{\boldsymbol{v}}}
\newcommand{\Bw}{{\boldsymbol{w}}}
\newcommand{\Bx}{{\boldsymbol{x}}}

\newcommand{\Bz}{{\boldsymbol{z}}}
\newcommand{\Btheta}{{\boldsymbol{\theta}}}

\newcommand{\BC}{{\boldsymbol{C}}}

\newcommand{\BH}{{\boldsymbol{H}}}

\newcommand{\BV}{{\boldsymbol{V}}}

\newcommand{\BX}{{\boldsymbol{X}}}

\newcommand{\Ct}{{\mathcal T}}

\def \endproof{\vrule height8pt width 5pt depth 0pt}
\newcommand{\nn}{\nonumber}

\def\refe#1{(\ref{#1})}

\begin{document}

\title[Mixed method for incompressible MHD]{New analysis of Mixed finite element methods for 
incompressible Magnetohydrodynamics}

\author{Yuchen Huang}
\address{
Advanced Institute of Natural Sciences, Beijing Normal University, 
Zhuhai 519087, P.R. China.  
}

\author{Weifeng Qiu} 
\address{Department of Mathematics, City University of Hong Kong, 83 Tat Chee Avenue, Kowloon, Hong Kong, China 
({\it weifeqiu@cityu.edu.hk}). }
\thanks{The work of W. Qiu was supported by a grant from the Research Grants 
Council of the Hong Kong Special Administrative Region, China 
(Project No. CityU 11302718). 
}

\author{Weiwei Sun}
\address{Research Center for Mathematics, Advanced Institute of Natural Science, Beijing Normal University,  Zhuhai, P.R. 
China} 
 \address{Guangdong Provincial Key Laboratory of Interdisciplinary Research and Application for Data Science (IRADS), BNU-HKBU United International College, Zhuhai, 519087, P.R.China (W.Sun, {\it maweiw@uic.edu.cn}).  
 } 
\thanks{
The work of Y. Huang and W. Sun was  partially supported 
by National 
Natural Science Foundation of China (12231003 and 12071040),  
Guangdong Provincial Key Laboratory IRADS (2022B1212010006, UIC-R0400001-22) and Guangdong Higher Education Upgrading Plan (UIC-R0400024-21). 
}


\begin{abstract}
This paper focuses on new error analysis of a class of mixed FEMs for stationary incompressible magnetohydrodynamics with 
the standard inf-sup stable velocity-pressure space in cooperation with Navier-Stokes equations and the N\'ed\'elec's 
edge element for the magnetic field. The methods have been widely used in various numerical simulations in the last 
several decades, while the existing analysis is not optimal due to the strong coupling of system and the pollution of 
the lower-order N\'ed\'elec's edge approximation in analysis. In terms of a newly  modified Maxwell projection we 
establish new and optimal error estimates. In particular, we prove that the method based on the commonly-used 
Taylor-Hood/lowest-order N\'ed\'elec's edge element is efficient and the method provides the second-order accuracy 
for numerical velocity.  Two numerical examples for the problem in both convex and nonconvex polygonal domains 
are presented,  which confirm our theoretical analysis. 
\end{abstract}



\maketitle

\section{Introduction}
\label{sec_intro}
Magnetohydrodynamics (MHD) is the study of the interaction between electrically conducting fluids and 
electromagnetic fields  \cite{Davidson01, GerbeauBris2006, Muller01}, such as  liquid metals, and salt water or electrolytes. 
Some more comprehensive discussion on the applications can be found in 
\cite{GerbeauBris2006, HughesYoung99, Moreau} and references therein.  In this paper, we consider the steady state incompressible MHD model on $\Omega \subset \mathbb{R}^d$, $d=2,3$, defined by 
\begin{subequations}
\label{mhd_eqs}
\begin{align}
\label{mhd_eq1}
- R_{e}^{-1}\Delta\Bu + (\Bu\cdot \nabla) \Bu + \nabla p - S (\nabla\times \Bb) \times \Bb & = \Bf \quad \text{ in } \Omega, \\
\label{mhd_eq2} 
S R_{m}^{-1} \nabla\times (\nabla\times \Bb) - S \nabla \times (\Bu \times \Bb) - \nabla r & = \Bg \quad \text{ in } \Omega, \\
\label{mhd_eq3}
\nabla \cdot \Bu & = 0 \quad \text{ in } \Omega, \\
\label{mhd_eq4}
\nabla\cdot \Bb & = 0 \quad \text{ in } \Omega, \\
\label{mhd_eq5}
\Bu & = \boldsymbol{0} \quad \text{ on } \partial\Omega, \\
\label{mhd_eq6}
\Bn \times \Bb & = \boldsymbol{0} \quad \text{ on } \partial\Omega, \\ 
\label{mhd_eq7}
r & = 0 \quad \text{ on } \partial\Omega, \\
\label{mhd_eq8}
 \int_{\Omega} p \, d\Bx & = 0, 
\end{align}
\end{subequations}
where $\Omega$ a simply-connected Lipschitz polygonal or polyhedral domain 
and $\Bn$ is the unit outward normal vector on $\partial\Omega$. The solution of the above system consists of the velocity $\Bu$, the pressure $p$, 
the magnetic field $\Bb$ and the Lagrange multiplier $r$ associated with the divergence constraint on the magnetic field $\Bb$. 
The above equations are characterized by three dimensionless parameters: 
the hydrodynamic Reynolds number $R_{e}$, the magnetic Reynolds number $R_{m}$ and the coupling number $S$. 
\cite{ArmeroSimo96, Davidson01, GerbeauBris2006} provide detailed discussion of these parameters and their typical values. 

Numerical methods and analysis for the MHD model have been investigated 
extensively in the last several decades, see 
\cite{Badia2013,Gerbeau2000,GerbeauBris2006,GreifLi2010,Gunzburger91,HZ,Hiptmair_LMZ, HuMaXu, 
HuXu2019, MeirSchmidt99, Shoetzau2004,ZhangHeYang} 
and references therein. The model is described by a coupled system of 
electrical fluid flows and electromagnetic fields, governed by 
Navier-Stokes and Maxwell type equations, respectively. Therefore, 
numerical methods for the MHD system are based on a combination of 
the approximation to Navier-Stokes equations and the approximation to Maxwell equations. 
Earlier works was mainly focused on the classical Lagrange type finite 
element approximation to the magnetic field $\Bb$. 
Analysis has been done by many authors \cite{DongHeZhang, Gerbeau2000, Gunzburger91, HZ, MeirSchmidt99}. 
\cite{Gunzburger91} firstly provides the existence, uniqueness, and optimal convergent finite element approximation 
to the MHD system with nonhomogeneous boundary conditions. 
Instead of assuming the source terms $\boldsymbol{f}$ and $\boldsymbol{g}$ are small enough, 
the analysis in \cite{Gunzburger91} only requires that $\Vert \Bu\Vert_{H^{\frac{1}{2}}(\partial \Omega)}$ is small enough
(see \cite[($4.19$)]{Gunzburger91}). 
A more popular approximation to Maxwell equations is 
the $H(\text{curl})$-conforming N\'ed\'elec's edge element methods, which 
have been widely used in many engineering areas. 
It is well-known that Lagrange type approximation may produce wrong numerical solutions for Maxwell equations in a nonconvex polyhedral domain, 
(see \cite{Amrouche-BDG, Costabel_Dauge}). 
For the MHD system, a class of mixed finite element methods  
was first presented by Sch{\"{o}}tzau \cite{Shoetzau2004}, where the hydrodynamic system 
is discretized by standard inf-sup stable velocity-pressure space pairs and the magnetic system by a mixed 
approach using N\'ed\'elec's elements of the first kind.  
Error estimates of methods were presented and the problem was considered in general domains.  
Subsequently, numerous efforts have been made with the N\'ed\'elec FE approximation 
\cite{Jin2016,LiZheng2017, PagliantiniThesis, Phillips2016,
Wathen2020, Wathen2017,ZhangHeYang} 
and the analysis has been extended to many different models and approximations 
\cite{DingLongMao2020, DingMao2020, GaoSun2015, Hiptmair2018, LSBS, Phillips2015}. 
For a convex polyhedral domain, the main result given in \cite{Shoetzau2004} is the following error estimate 
\begin{align} 
\| \Bu - \Bu_h \|_{H^1(\Omega)} + \| \Bb - \Bb_h \|_{H(\text{curl}, \Omega)} \le 
C (h^l + h^k) 
\label{error-lk}  
\end{align} 
for the method with the approximation accuracy $O(h^l)$ 
for hydrodynamic  variables and the approximation accuracy $O(h^k)$ for 
the magnetic field $\Bb$. 
By \refe{error-lk},  one has to take the combination with $k=l$ to achieve an optimal convergence rate. However, 
the method with $k<l$ is more popular since 
high-order N\'ed\'elec's edge elements are more complicated in implementation 
and extremely time-consuming in computation. 
In particular, the method based on the combination of  the 
Taylor-Hood element and the 
lowest-order N\'ed\'elec's edge element has been frequently used in applications
 and numerical simulations have been done 
extensively \cite{GaoQiu2019, SchneebeliSchoetzau2003,Wathen2020,Wathen2017}. 
 In this case,  $k=1$ and $l=2$, the error estimate \refe{error-lk} reduces to 
\begin{align} 
\| \Bu - \Bu_h \|_{H^{1}(\Omega)} + \| \Bb - \Bb_h \|_{H(\text{curl}, \Omega)} \le 
C (h^2 + h) \, . 
\label{error-1} 
\end{align} 
One can see from \refe{error-1} that the accuracy of numerical velocity is only of the first-order, which is not optimal in the traditional sense and also,  not a good indication 
for the commonly-used method. 
It was assumed that the accuracy of the velocity is polluted by 
the lower-order N\'ed\'elec's edge finite element approximation.  
This is a common question in many applications when FEMs with combined  approximations of different orders is used for 
a strongly coupled system. 
The main purpose of this paper is to establish the optimal error estimate 
 \begin{align} 
\| \Bu - \Bu_h \|_{H^{1}(\Omega)} 
\le 
C (h^l + h^{k+1}) 
\end{align} 
for the standard combination, which shows that the numerical velocity 
is of one-order higher accuracy than given in previous analysis for the case $k < l$ and which implies 
the second-order accuracy 
\begin{align} 
\| \Bu - \Bu_h \|_{H^{1}(\Omega)} 
\le C h^2  
\end{align} 
for the combination of Taylor-Hood element and the 
lowest-order N\'ed\'elec's edge element of the first type. 
Our analysis is based on a new modified Maxwell projection. 
In terms of the projection and the error estimate in a negative norm, 
a more precise analysis is presented in this paper. 
The analysis shows clearly that the  
mixed method with the Taylor-Hood/lowest-order N\'ed\'elec's edge element approximations is efficient and the method provides second-order 
 accuracy for numerical velocity. The lower-order approximation to 
 the magnetic field $\Bb$ does not influence the accuracy of numerical solution of Navier-Stokes equations.


The rest of the paper is organized as follows. In Section 2 we first provide 
the variational formulation and the mixed method for the MHD model and 
some existing results and then, we present our main theorem for an optimal error estimate of the method. 
To prove it, we introduce a modified Maxwell 
projection and establish its approximation properties in Section 3. 
In terms of this projection, we present our theoretical analysis. 
In section 4, we provide numerical experiments to confirm our theoretical analysis and show the efficiency of the method.

\section{Mixed FEMs and main results} 

\subsection{Mixed FEMs} 
\label{sec_methods}
To introduce the mixed method, we adopt the notations and 
norms used in \cite{Shoetzau2004,ZhangHeYang}. We denote some standard 
vector and scalar function spaces by 
\begin{align*}
H(\text{curl}, \Omega) = & \{ \Bc \in [ L^{2}(\Omega) ]^d: \nabla \times \Bc \in [ L^{2}(\Omega) ]^d \}, \\
H_{0}(\text{curl}, \Omega) = &  \{ \Bc \in H(\text{curl}, \Omega): \Bn \times \Bc |_{\partial\Omega}  
= \boldsymbol{0} \}, \\
\BX = & \{ \Bc \in H_{0}(\text{curl}, \Omega):  \nabla\cdot \Bc = 0\}, \\
H(\text{div}, \Omega) = & \{ \Bc \in [ L^{2}(\Omega) ]^d: \nabla \cdot \Bc \in [ L^{2}(\Omega) ]^d  \}, \\
L_{0}^{2}(\Omega) = & \{ q \in L^{2}(\Omega) : \int_{\Omega} q d\Bx = 0 \}, \\
H^{-1}(\Omega)  = & \left( H_{0}^{1}(\Omega) \right)^{*}. 
\end{align*}
For any $(\Bv, \Bc)\in [H_{0}^{1}(\Omega)]^d \times H(\text{curl}, \Omega)$, we define 
\begin{align}
\label{combined_norm1}
\Vert (\Bv, \Bc) \Vert^{2} := \Vert  \nabla \Bv \Vert_{L^{2}(\Omega)}^{2} + S\Vert  \Bc \Vert_{H(\text{curl}, \Omega)}^{2}. 
\end{align}
Moreover, we denote some bilinear or trilinear forms by 
\begin{align*}
& a_{s}(\Bu, \Bv) = R_{e}^{-1} (\nabla \Bu, \nabla \Bv)_{\Omega}, \\ 
& a_{m}(\Bb, \Bc) = S R_{m}^{-1} (\nabla \times \Bb, \nabla \times \Bc)_{\Omega}, \\
& c_{0}(\Bw; \Bu, \Bv) = \frac{1}{2} (\Bw \cdot \nabla \Bu, \Bv)_{\Omega} - \frac{1}{2} (\Bw \cdot \nabla \Bv, \Bu)_{\Omega}, \\
& c_{1}(\Bd; \Bv, \Bc) = S ( (\nabla \times \Bc) \times \Bd, \Bv)_{\Omega} 
= - S ( \Bv \times \Bd, \nabla \times \Bc)_{\Omega}, 
\end{align*}
for any $\Bu, \Bv \in [H_{0}^{1}(\Omega)]^d$ and any $\Bb, \Bc, \Bd \in H_{0}(\text{curl}, \Omega)$ 
with $\Bd \in [ L^{3}(\Omega) ]^d$.

The exact solution $(\Bu, \Bb, p, r)$ of the MHD system (\ref{mhd_eqs}) satisfies the variational formulation 
\begin{subequations}
\label{mhd_var_eqs}
\begin{align}
\label{mhd_var_eq1}
a_{s}(\Bu, \Bv) + c_{0}(\Bu; \Bu, \Bv) -  c_{1}(\Bb; \Bv, \Bb) -  (p, \nabla \cdot \Bv)_{\Omega} 
= & (\Bf, \Bv)_{\Omega}, \\
\label{mhd_var_eq2}
a_{m}(\Bb, \Bc) + c_{1} (\Bb; \Bu, \Bc) - (\nabla r, \Bc)_{\Omega} = & (\Bg, \Bc)_{\Omega}, \\
\label{mhd_var_eq3}
(\nabla\cdot \Bu, q)_{\Omega} = & 0, \\
\label{mhd_var_eq4} 
(\Bb, \nabla s)_{\Omega} = & 0, 
\end{align}
\end{subequations}
for any $(\Bv, \Bc, q, s) \in [H_{0}^{1}(\Omega)]^d \times H_{0}(\text{curl}, \Omega) \times L_{0}^{2}(\Omega) 
\times H_{0}^{1}(\Omega)$. 

Let $\mathcal{T}_{h}$ denote a quasi-uniform  conforming triangulation of $\Omega$.  
On this triangulation, we define several finite element spaces by 
\begin{align*}
\BV_{h}^{l} &:=  [ H_{0}^{1}(\Omega) \cap  P_{l}(\Ct_{h})]^d, \\
\quad Q_{h}^{l} &:=   P_{l-1}(\Ct_{h}) \cap H^{1}(\Omega)  \cap L_{0}^{2}(\Omega), \\
\quad \BC_{h}^{k} &:= \{ \Bc_{h} \in H_{0}(\text{curl}, \Omega): \Bc_{h} |_{K}  \in [P_{k-1}(K)]^d \oplus D_{h}^{k}(K),  
\forall K \in \Ct_{h}\},\\
\quad S_{h}^{k} &:=  H_{0}^{1}(\Omega) \cap P_{k}(\Ct_{h})
\end{align*}
for $l \geq 2$ and $k \geq 1$, where $P_{l}(\Ct_{h}) = \{ w \in L^{2}(\Omega): w|_{K} \in P_{l}(K), 
\forall K \in \Ct_{h}  \}$, 
$D_{h}^{k}(K) = \{ \Bp \in [\tilde{P}_{k}(K)]^d :  \Bp (\Bx) \cdot \Bx = 0, \forall \Bx \in K  \}$ 
and 
$\tilde{P}_{k}(K)$ is the collection of  the $k$-th order homogeneous polynomials in $P_{k}(K)$. 
 $ \BC_{h}^{k}$ is actually 
the $k$-th order first type of N\'ed\'elec's edge element space. 

The mixed method in \cite{Shoetzau2004,ZhangHeYang} seeks an approximation $(\Bu_{h}, \Bb_{h}, p_{h}, r_{h})
\in \BV_{h}^{l}\times \BC_{h}^{k}\times Q_{h}^{k}\times S_{h}^{l}$ 
to the exact solution $(\Bu,\Bb, p, r)$ satisfying 
the following weak formulation:
\begin{subequations}
\label{MFEM_mhd}
\begin{align}
\label{MFEM_mhd_eq1}
a_{s}(\Bu_{h}, \Bv_h) + c_{0}(\Bu_{h}; \Bu_{h}, \Bv_h) - c_{1}(\Bb_{h}; \Bv_h, \Bb_{h}) 
 -  (p_{h}, \nabla\cdot\Bv_h)_{\Omega} & = (\Bf, \Bv_h)_{\Omega}, \\
\label{MFEM_mhd_eq2} 
a_{m}(\Bb_{h}, \Bc_h) + c_{1}(\Bb_{h}; \Bu_{h}, \Bc_h) - (\nabla r_{h}, \Bc_h)_{\Omega} & = (\Bg, \Bc_h)_{\Omega}, \\ 
\label{MFEM_mhd_eq3}
(\nabla\cdot \Bu_{h}, q_h)_{\Omega} &= 0,\\
\label{MFEM_mhd_eq4} 
(\Bb_{h}, \nabla s_h)_{\Omega} &= 0,
\end{align}
\end{subequations}
for all $(\Bv_h, \Bc_h, q_h, s_h)\in \BV_{h}^{l}\times \BC_{h}^{k}\times Q_{h}^{l}\times S_{h}^{k}$. 

The paper is focused on optimal error estimates of the mixed method defined in \refe{MFEM_mhd}. 
Iterative algorithms for solving the nonlinear algebraic system and 
their convergences were studied by several authors \cite{DongHeZhang, Phillips2016, Wathen2020,Wathen2017, ZhangHeYang} 
and numerical simulations 
on various practical models can be found in literature \cite{Badia2013, GaoQiu2019,  SchneebeliSchoetzau2003}.  
Analysis presented in this paper can be extended to many other mixed methods.

\subsection{Auxiliary results}
The mixed method defined in 
\refe{MFEM_mhd_eq1}-\refe{MFEM_mhd_eq4} was analyzed by several authors. 
In this subsection, we provide some existing results which shall be used in 
our analysis. 

Mimicking the space $\BX$ defined at the beginning of Section~\ref{sec_methods}, we introduce 
\begin{align}
\label{X_h_space}
\BX_{h} : = \{ \Bc_{h} \in \BC_{h}^{k}: (\Bc_{h}, \nabla s_{h})_{\Omega} = 0, \forall s_{h} \in S_{h}^{k}  \}.
\end{align}

\begin{lemma}
\label{lemma_poincare_embedding} 
(see \cite[($2.2$, $2.3$, $2.4$)]{ZhangHeYang})
There exist positive constants $\lambda_{0}, \lambda_{1}^{*}, \lambda_{1}$ and $\lambda_{2}$ such that 
\begin{subequations}
\label{poincare_embedding_eqs}
\begin{align}
\label{poincare_embedding_eq1}
& \lambda_{0} \Vert \Bc \Vert_{H(\text{curl}, \Omega)} \leq \Vert \nabla \times \Bc \Vert_{L^{2}(\Omega)}, 
\quad \forall \Bc \in \BX, \\
\label{poincare_embedding_eq2}
& \Vert v\Vert_{L^{3}(\Omega)} \leq \lambda_{1}^{*} \Vert \nabla v\Vert_{L^{2}(\Omega)}, \quad 
\Vert v\Vert_{L^{6}(\Omega)} \leq \lambda_{1} \Vert \nabla v\Vert_{L^{2}(\Omega)}, \quad \forall v \in H_{0}^{1}(\Omega), \\ 
\label{poincare_embedding_eq3}
& \Vert \Bc\Vert_{L^{3}(\Omega)} \leq \lambda_{2} \Vert \nabla \times \Bc\Vert_{L^{2}(\Omega)},  \quad \forall \Bc \in \BX.
\end{align}
\end{subequations}
\end{lemma}

\begin{lemma}
(\cite[Lemma~$2.1$]{ZhangHeYang})
\label{lemma_terms_ineqs}
It holds that 
\begin{subequations}
\label{coercive_terms_ineqs}
\begin{align}
\label{coercive_terms_ineq1}
a_{s}(\Bw, \Bv) + a_{m}(\Bd, \Bc) \leq & \max \{R_{e}^{-1}, R_{m}^{-1} \} \Vert (\Bw, \Bd)\Vert \cdot \Vert (\Bv, \Bc) \Vert, \\
\nonumber
& \qquad \qquad \forall (\Bw, \Bd), (\Bv, \Bc) \in [ H_{0}^{1}(\Omega) ]^d \times H_{0}(\text{curl}, \Omega), \\
\label{coercive_terms_ineq2}
a_{s}(\Bv, \Bv) + a_{m} (\Bc, \Bc) \geq & \min (R_{e}^{-1}, R_{m}^{-1} \lambda_{0}) \Vert (\Bv, \Bc) \Vert^{2}, \\
\nonumber
& \qquad \qquad \forall (\Bv, \Bc) \in [H_{0}^{1}(\Omega)]^d \times \BX 
\end{align}
\end{subequations}
and 
\begin{subequations}
\label{couple_trems_ineqs}
\begin{align}
\label{couple_terms_ineq1}
\sup_{\Bw, \tilde{\Bv}, \Bv \in [H_{0}^{1}(\Omega)]^d} \dfrac{c_{0}(\Bw; \tilde{\Bv}, \Bv)}{\Vert \nabla \Bw \Vert_{L^2(\Omega)}
\Vert \nabla \tilde{\Bv} \Vert_{L^2(\Omega)} \Vert  \nabla \Bv \Vert_{L^2(\Omega)}} = &  \lambda_{1} \lambda_{1}^{*}, \\
\label{couple_terms_ineq2}
\sup_{\Bd \in \BX, \Bv \in [H_{0}^{1}(\Omega)]^d, \Bc \in H_{0}(\text{curl}, \Omega)} 
\dfrac{c_{1}(\Bd; \Bv, \Bc)}{\Vert \Bd\Vert_{H(\text{curl},\Omega)} \Vert \nabla \Bv\Vert_{L^{2}(\Omega)} 
\Vert  \Bc \Vert_{H(\text{curl}, \Omega)} } = & S \lambda_{1} \lambda_{2}, \\
\label{couple_terms_ineq3}
\sup_{\Bd\in \BX, \tilde{\Bc}, \Bc \in H_{0}(\text{curl},\Omega), \Bw, \tilde{\Bv}, \Bv \in [H_{0}^{1}(\Omega)]^d} 
\dfrac{c_{0}(\Bw; \tilde{\Bv}, \Bv) -c_{1}(\Bd; \Bv, \tilde{\Bc}) + c_{1}(\Bd; \tilde{\Bv}, \Bc) }{\Vert (\Bw, \Bd)\Vert 
\cdot \Vert (\tilde{\Bv}, \tilde{\Bc})\Vert \cdot \Vert (\Bv, \Bc)\Vert } = & \widehat{N}_1\\
\nonumber
 := & \sqrt{2} \lambda_1 \max \{\lambda_1^{*}, \lambda_2 \}.
\end{align}
\end{subequations}
Here the constants $\lambda_{0}, \lambda_{1}^{*}, \lambda_{1}$ and $\lambda_{2}$ 
are introduced 
in Lemma~\ref{lemma_poincare_embedding}.
\end{lemma}

For any $\lambda > 0$, we define 
\begin{align} 
\label{eta_lambda}
\eta(\lambda): = \dfrac{\left( \Vert \Bf\Vert_{L^{2}(\Omega)} + S^{-1} \Vert \Bg\Vert_{L^{2}(\Omega)} \right)}{\left( 
\min \{ R_{e}^{-1}, R_{m}^{-1} \lambda \}  \right)^{2}}.
\end{align}
The well-posedness of the MHD system \refe{mhd_eqs} is given in the following lemma 
and the proof can be found in \cite{Shoetzau2004, ZhangHeYang}. 

\begin{lemma}
\label{lemma_mhd_wellposedness}
Suppose that 
\begin{align} 
\label{cond-1} 
\widehat N_1 \eta(\lambda_0) < 1, 
\end{align}  
where $\widehat{N}_{1}$ is introduced in (\ref{couple_terms_ineq3}) and 
and $\eta(\lambda_{0})$ is defined as (\ref{eta_lambda}) with $\lambda = \lambda_{0}$. 
Then the MHD system (\ref{mhd_eqs}) admits a unique solution 
$(\Bu, \Bb, p, r) \in [H_{0}^{1}(\Omega)]^d \times H_{0}(\text{curl}, \Omega) \times L_{0}^{2}(\Omega) 
\times H_{0}^{1}(\Omega)$ satisfying 
\begin{align}
\Vert (\Bu, \Bb) \Vert \leq 
 \eta(\lambda_0) \min \{ R_{e}^{-1}, R_{m}^{-1} \lambda_{0} \} \, . 
\label{ub-bound} 
\end{align}
\end{lemma}

\begin{lemma}
(\cite[Lemma~$7.20$]{Monk2003}, \cite[Lemma~$5.1$]{QiuShiMHD1})
\label{lemma_discrete_curl_embedding}
There exist positive constants $\lambda_{0}^{*}, \lambda_{2}^{*}$ independent of $h$ such that 
\begin{align}
\label{discrete_curl_embedding_ineq}
\lambda_{0}^{*} \Vert \Bc_h\Vert_{H(\text{curl},\Omega)} \leq \Vert \nabla \times \Bc_h\Vert_{L^{2}(\Omega)}, \quad
\Vert \Bc_h \Vert_{L^{3}(\Omega)} \leq \lambda_{2}^{*} \Vert \nabla \times \Bc_h \Vert_{L^{2}(\Omega)}, 
\quad \forall \Bc_h \in \BX_{h}
\end{align}
where the finite element space $\BX_{h}$ is introduced in (\ref{X_h_space}).
\end{lemma}

The well-posedness of the finite element system and error estimates of finite element solutions were presented  in  
\cite{Shoetzau2004, ZhangHeYang}. With the above lemma, 
the well-posedness with a slightly weak condition is given in the following lemma.  
The proof follows those given in \cite{Shoetzau2004, ZhangHeYang} and is 
omitted here. 

\begin{lemma}
\label{thm_mfem_wellposedness}
Supposed that 
\begin{align} 
\label{cond-2} 
\widehat N_2 \eta(\lambda_{0}^{*}) < 1, 
\end{align} 
where $\widehat N_{2} = \sqrt{2} \lambda_{1} \max \{ \lambda_{1}^{*}, \lambda_{2}^{*} \}$ 
and $\eta(\lambda_{0}^{*})$ is defined as (\ref{eta_lambda}) with $\lambda = \lambda_{0}^{*}$.   
Then the mixed finite element system (\ref{MFEM_mhd}) admits a unique solution satisfying 
\begin{align}
\label{mfem_bound}
\Vert (\Bu_{h}, \Bb_{h}) \Vert \leq \eta(\lambda_0^*) 
\min \{ R_{e}^{-1}, R_{m}^{-1} \lambda_0^* \}   .
\end{align}
Here the constants $\lambda_{0}^{*}, \lambda_{2}^{*}$ are introduced 
in Lemma~\ref{lemma_discrete_curl_embedding}, while $\lambda_{1}, \lambda_{1}^{*}$ are defined in 
Lemma~\ref{lemma_poincare_embedding}.

In addition, 
\begin{align}
\label{mfem_conv}
\Vert \Bu - \Bu_{h}\Vert_{H^{1}(\Omega)} + \Vert \Bb - \Bb_{h} \Vert_{H(\text{curl}, \Omega)} 
\leq C (h^{l} + h^{k}) \, . 
\end{align}
\end{lemma}

\begin{remark}
We can see that the condition~(\ref{cond-2}) and discrete inverse inequality
implies the condition~(\ref{cond-1}) and from \refe{mfem_conv} that 
\begin{align} 
\| \Bu_h \|_{W^{1,p}(\Omega)} \le C \qquad \mbox{ for } p \le 6 \, . 
\label{bound}
\end{align}  
\end{remark}

\subsection{Main results} 

Under the assumptions of Lemma \ref{lemma_mhd_wellposedness},  
the MHD system \refe{mhd_eqs} is well-posed. 
We further assume that 
the solution satisfies the following regularity condition: 
 there exists a positive constant $K$, such that 
\begin{align}
\Vert \Bb \Vert_{W^{1, d^{+}}(\Omega)}  + 
 \Vert \nabla \times\Bb \Vert_{W^{1, d^{+}}(\Omega)} +   
\Vert \Bu\Vert_{H^{l+1}(\Omega)} & + 
 \Vert p\Vert_{H^{l}(\Omega)} +  \Vert \Bb\Vert_{H^{k}(\Omega)} 
 \nn \\ 
 & +   
\Vert \nabla \times \Bb\Vert_{H^{k}(\Omega)} +  \Vert r\Vert_{H^{k+1}(\Omega)}
 \leq K.
 \label{assump_K}
\end{align}
Here $d^{+}$ denotes a constant strictly bigger than $d$.

Our main result is the following Theorem~\ref{thm_main_result}.
\begin{theorem}
\label{thm_main_result}
We assume that the domain $\Omega$ is a convex polygon or polyhedra in $\mathbb{R}^d$ and the conditions 
(\ref{cond-2}, \ref{assump_K}) hold. 
Then the mixed finite element system (\ref{MFEM_mhd}) admits a unique solution 
and there exists $h_0 > 0$ such that when $h \le h_0$, 
\begin{align}
\label{main} 
& \Vert \Bu - \Bu_{h}\Vert_{H^{1}(\Omega)} + \Vert \tilde{\Bb}_{h} - \Bb_{h} \Vert_{H(\text{curl}, \Omega)}
\leq C_{1} \left( h^{l} + h^{k+1} \right), \\
\nonumber 
& \Vert \Bb - \Bb_{h}\Vert_{H(\text{curl}, \Omega)} \leq C_{1} \left( h^{l} + h^{k} \right),
\end{align} 
where $C_1$ is a positive constant depending upon the physical parameters 
$S, R_m, R_e$, the domain $\Omega$ and the constant $K$ introduced in (\ref{assump_K}). 
Here $\tilde{\Bb}_{h} \in \BC_{h}^{k}$ is a projection of $(\Bb, r)$ defined below in (\ref{Maxwell_projection}). 
\end{theorem}

\begin{corollary} 
\label{main-coro}
Under the assumptions of Theorem~\ref{thm_main_result}, it holds that 
\begin{align}
\label{main-2} 
\Vert \Bu - \Bu_{h}\Vert_{L^2(\Omega)} + \Vert \Bb - \Bb_{h} \Vert_{H^{-1}(\Omega)} 
\leq C_{2} \left( h^{l+1} + h^{k+1} \right),
\end{align} 
where $C_2$ is a positive constant depending upon the physical parameters 
$S, R_m, R_e$, the domain $\Omega$ and the constant 
$K$ introduced in (\ref{assump_K}). 
\end{corollary} 
\vskip0.1in

{\em Remarks. } For the Taylor-Hood/lowest-order 
N\'ed\'elec's edge element of the first type,  $l=2$ and $k=1$. From the above theorem, 
one can see that the frequently-used  mixed method 
provides the second-order accuracy for the numerical velocity, while only the first-order accuracy 
was presented in previous analyses.  For the MINI/lowest-order 
N\'ed\'elec's edge element of the first type,  $l=k=1$ and the optimal error estimate of the second-order in $L^2$-norm is shown in \refe{main-2}.  
For simplicity, hereafter we denote by $C_K$ a generic positive constant which 
depends upon $K$.

\section{Analysis} 
Before proving our main results, we present a modified Maxwell projection in the following subsection, which 
plays a key role in the proof of Theorem \ref{thm_main_result}. 

\subsection{Projections} 
Let $(\tilde{\Bu}_{h}, \tilde{p}_{h}) \in \BV_{h}^{l} \times Q_{h}^{l}$ be the standard Stokes projection of $(\Bu, p)$ defined by 
\begin{subequations}
\label{Stokes_projection}
\begin{align}
\label{Stokes_projection_eq1}
a_{s}(\Bu - \tilde{\Bu}_{h}, \Bv_{h}) - (p - \tilde{p}_{h}, \nabla\cdot \Bv_{h})_{\Omega} = & 0, 
\quad \forall \Bv_{h} \in \BV_{h}^{l} \\
\label{Stokes_projection_eq2}
(\nabla\cdot (\Bu - \tilde{\Bu}_{h}), q_{h})_{\Omega} = & 0, \quad \forall q_{h} \in Q_{h}^{l}. 
\end{align}
\end{subequations}
Let $(\tilde{\Bb}_{h}, \tilde{r}_{h}) \in \BC_{h}^{k} \times S_{h}^{k}$ be the modified Maxwell projection of $(\Bb, r)$ defined by 
\begin{subequations}
\label{Maxwell_projection}
\begin{align}
\label{Maxwell_projection_eq1}
a_{m} (\Bb - \tilde{\Bb}_{h}, \Bc_{h}) + c_{1}(\Bb - \tilde{\Bb}_{h}; \Bu, \Bc_{h})
 - (\nabla (r - \tilde{r}_{h}), \Bc_{h})_{\Omega} = & 0, 
\quad \forall \Bc_{h} \in \BC_{h}^{k}, \\
\label{Maxwell_projection_eq2}
(\Bb - \tilde{\Bb}_{h}, \nabla s_{h})_{\Omega} = & 0, \quad \forall s_{h} \in S_{h}^{k}.
\end{align}
\end{subequations}

With the Stokes projection (\ref{Stokes_projection}) and the modified Maxwell projection (\ref{Maxwell_projection}), 
we define an error splitting by 
\begin{subequations}
\label{err_splits}
\begin{align}
\label{err_split1}
\Bu - \Bu_{h} = & (\Bu - \tilde{\Bu}_{h}) + (\tilde{\Bu}_{h} - \Bu_{h}) := \xi_{\Bu} + e_{\Bu},\\  
\label{err_split2}
\Bb - \Bb_{h} = & (\Bb - \tilde{\Bb}_{h}) + (\tilde{\Bb}_{h} - \Bb_{h}) := \xi_{\Bb} + e_{\Bb}, \\
\label{err_split3}
p - p_{h} = & (p - \tilde{p}_{h}) + (\tilde{p}_{h} - p_{h}) := \xi_{p} + e_{p},\\ 
\label{err_split4} 
r - r_{h} = & (r - \tilde{r}_{h}) + (\tilde{r}_{h} - r_{h}) := \xi_{r} + e_{r}.
\end{align}
\end{subequations}

Moreover, 
we denote by $(\widehat{\Bb}_{h}, \widehat{r}_{h}) \in \BC_{h}^{k} \times S_{h}^{k}$ 
 the standard Maxwell 
projection defined by 
\begin{subequations}
\label{maxwell_project_standard}
\begin{align}
a_{m} (\Bb - \widehat{\Bb}_{h}, \Bc_{h})  - (\nabla (r - \widehat{r}_{h}), \Bc_{h})_{\Omega} = & 0, 
\quad \forall \Bc_{h} \in \BC_{h}^{k}, \\
(\Bb - \widehat{\Bb}_{h}, \nabla s_{h})_{\Omega} = & 0, \quad \forall s_{h} \in S_{h}^{k}.
\end{align}
\end{subequations} 

By classic finite element theory, we have the error estimates
\begin{align}
\label{s-proj-error} 
  \| \xi_{\Bu} \|_{L^{2}(\Omega)}  + h (\| \nabla \xi_{\Bu} \|_{L^{2}(\Omega)}  
+ \| \xi_{p} \|_{L^{2}(\Omega)}) 
 \leq  C h^{l+1} ( \| \Bu \|_{H^{l+1}(\Omega)} + \| p \|_{H^{l}(\Omega)} )
\end{align} 
 for the Stokes projection in \refe{Stokes_projection} when $\Omega$ is convex and 
\begin{align}
 \Vert \nabla \times (\widehat{\Bb}_{h} - \Bb) \Vert_{L^{2}(\Omega)} 
& +  \Vert \widehat{\Bb}_{h} - \Bb \Vert_{L^{2}(\Omega)} 
+ \| \nabla (\widehat{r}_{h} - r) \|_{L^2(\Omega)}  \nn \\ 
& \leq C h^{\min (k, s)} \left(\Vert \Bb\Vert_{H^{s}(\Omega)} + \Vert \nabla \times \Bb \Vert_{H^{s}(\Omega)} 
+ \Vert r\Vert_{H^{s+1}(\Omega)} \right) 
\label{M-projection-error} 
\end{align}
for the Maxwell projection in (\ref{maxwell_project_standard}). 
Here $s > 0$.

In order to provide the error estimates for the modified Maxwell projection, we consider 
$(\Bz, \phi) \in H_{0}(\text{curl}, \Omega) \times H_{0}^{1}(\Omega)$ satisfying 
\begin{subequations}
\label{dual_problem_eqs}
\begin{align}
\label{dual_problem_eq1}
S R_{m}^{-1} \nabla \times (\nabla \times \Bz) + S \Bu \times (\nabla \times \Bz) - \nabla \phi = & \Btheta, \\
\label{dual_problem_eq2}
\nabla \cdot \Bz = & 0,
\end{align}
\end{subequations}
where $\Btheta \in [ L^{2}(\Omega) ]^d$. 
The well-posedness of the above system is presented in the following theorem.  

\begin{theorem}
\label{thm_dual_problem} Suppose that (\ref{cond-1}) holds. 
Then the  system (\ref{dual_problem_eqs}) has a unique solution.  
If we further assume $\Omega$ is a convex polygon or polyhedra, $\Bu \in [L^{ \infty}(\Omega)\cap W^{1,3}(\Omega)]^d$
 and $\Btheta \in H(\text{div}, \Omega)$, then 
\begin{align}
\label{dual_problem_estimate}
\Vert \Bz\Vert_{H^{1}(\Omega)} + \Vert \nabla \times \Bz \Vert_{H^{1}(\Omega)} + \Vert \phi \Vert_{H^{2}(\Omega)} 
\leq C_{3} \Vert \Btheta \Vert_{H(\text{div},\Omega)},
\end{align}
where $C_3$ is a positive constant depending upon the physical parameters 
$S, R_m$, the domain $\Omega$ and the constant $K$ introduced in (\ref{assump_K}). 
\end{theorem}

\begin{proof}
To show the uniqueness of the solution of the system (\ref{dual_problem_eqs}), 
we only consider the corresponding homogeneous system with 
$\Btheta = \boldsymbol{0}$. 
By standard energy argument, we have 
\begin{align*}
S R_{m}^{-1} \Vert \nabla \times \Bz \Vert_{L^{2}(\Omega)}^{2} + S (\Bu \times (\nabla \times \Bz), \Bz)_{\Omega} = 0 \, . 
\end{align*}
By noting (\ref{dual_problem_eq2}), we see that $\Bz \in \BX$. 
According to (\ref{poincare_embedding_eq2}) and (\ref{poincare_embedding_eq3}), we further have 
\begin{align*}
R_{m}^{-1} \Vert \nabla \times \Bz \Vert_{L^{2}(\Omega)}^{2} \leq  \lambda_{1} \lambda_{2} 
\Vert  \nabla \Bu \Vert_{L^{2}(\Omega)} \Vert \nabla \times \Bz \Vert_{L^{2}(\Omega)}^{2}
\end{align*}
which with the condition (\ref{cond-1}) shows that 
$\Vert \nabla \times \Bz \Vert_{L^2} =0$ and by noting $\Bz \in \BX$,  
 $\Bz = \boldsymbol{0}$ in $\Omega$. 
By (\ref{dual_problem_eq1}) and the assumption $\Btheta = \boldsymbol{0}$, 
we obtain $\phi = 0$ in $\Omega$. 
So the system (\ref{dual_problem_eqs}) has a unique solution.

Again by standard energy argument, we have 
\begin{align*}
S R_{m}^{-1} \Vert \nabla \times \Bz \Vert_{L^{2}(\Omega)}^{2} + S \left(\Bu \times (\nabla \times \Bz), \Bz \right)_{\Omega} 
= (\Btheta, \Bz)_{\Omega}. 
\end{align*}
By (\ref{poincare_embedding_eq2}, \ref{poincare_embedding_eq3}, \ref{cond-1}) and noting 
the fact that $\Bz \in \BX$, we see that 
\begin{align*}
\frac{1}{2} R_{m}^{-1} \Vert \nabla \times \Bz \Vert_{L^{2}(\Omega)}^{2} \leq S^{-1}\Vert \Btheta\Vert_{L^{2}(\Omega)} 
\Vert \Bz\Vert_{L^{2}(\Omega)} \leq S^{-1} \lambda_{0}^{-1} \Vert \Btheta\Vert_{L^{2}(\Omega)} 
\Vert \nabla \times \Bz\Vert_{L^{2}(\Omega)}  
\end{align*} 
where we have noted  (\ref{poincare_embedding_eq1}). It follows that 
\begin{align}
\label{dual_problem_bound1}
\Vert \Bz \Vert_{H(\text{curl}, \Omega)} \leq C \Vert \Btheta\Vert_{L^{2}(\Omega)}
\end{align}
and therefore, 
\begin{align}
\label{dual_problem_bound2}
\Vert \Bu \times (\nabla \times \Bz)\Vert_{L^{2}(\Omega)} \leq \Vert \Bu\Vert_{L^{\infty}(\Omega)} 
\Vert \nabla \times \Bz \Vert_{L^{2}(\Omega)} \leq C \Vert \Btheta\Vert_{L^{2}(\Omega)}.
\end{align} 

When $\Omega$ is a convex polygon (or polyhedra), we have 
$$ 
\| \Bz \|_{H^1(\Omega)} \le C (\| \nabla \times \Bz \|_{L^2(\Omega)} 
+ \| \nabla \cdot \Bz \|_{L^2(\Omega)} ) \, . 
$$ 
From the system (\ref{dual_problem_eqs}), (\ref{dual_problem_bound2}) and 
the  fact that $\phi \in H_{0}^{1}(\Omega)$, we can see that 
\begin{align}
\label{dual_problem_bound3}
\Vert \Bz\Vert_{H^{1}(\Omega)} + \Vert \nabla \times \Bz \Vert_{H^{1}(\Omega)} + \Vert \phi \Vert_{H^{1}(\Omega)}  
 \leq C \Vert \Btheta\Vert_{L^{2}(\Omega)}
\end{align}
and then, 
\begin{align}
\label{dual_problem_bound4} 
\Vert \nabla\cdot \left( \Bu \times (\nabla \times \Bz) \right) \Vert_{L^{2}(\Omega)} 
\leq C \left( \Vert \nabla \Bu \Vert_{L^{3}(\Omega)} + \Vert \Bu\Vert_{L^{\infty}(\Omega)} \right) 
\Vert  \nabla \times \Bz \Vert_{H^{1}(\Omega)} \leq C \Vert \Btheta \Vert_{L^{2}(\Omega)}. 
\end{align}
Moreover, by taking divergence on the both sides of (\ref{dual_problem_eq1}), we get the equation  
\begin{align*}
- \Delta \phi = \nabla \cdot \Btheta  - S  \nabla\cdot \left( \Bu \times (\nabla \times \Bz) \right).
\end{align*}
By noting (\ref{dual_problem_bound4}) and the assumption that $\Omega$ is a convex polyhedral,  
\begin{align}
\label{dual_problem_bound5}
\Vert  \phi \Vert_{H^{2}(\Omega)} \leq C \Vert \Btheta\Vert_{H(\text{div}, \Omega)}. 
\end{align}
(\ref{dual_problem_estimate}) follows from (\ref{dual_problem_bound3}) and (\ref{dual_problem_bound5}) and the proof is complete. 
\end{proof}
\vskip0.1in 

Now we present the error estimates of the modified Maxwell projection below. 

\begin{theorem}
\label{thm_Maxwell_projection}
We assume that $\Bu \in [H^{2}(\Omega)]^d$ 
and the condition (\ref{cond-2}) holds. 
Then the modified Maxwell projection (\ref{Maxwell_projection}) is well defined for 
any $(\Bb, r) \in \BX \times H_{0}^{1}(\Omega)$ and for any $s>0$, 
\begin{subequations}
\label{Maxwell_projection_props}
\begin{align}
\label{Maxwell_projection_prop1}
 \Vert \xi_{\Bb}\Vert_{H(\text{curl},\Omega)} + \Vert \nabla \xi_{r}\Vert_{L^{2}(\Omega)}  
\leq C_4 h^{\min (k, s)}  \, . 
\end{align} 
If we further assume $\Omega$ is a convex polygon or polyhedra, then 
\begin{align} 
\label{Maxwell_projection_prop2}
 & \Vert \xi_{\Bb}\Vert_{L^{3}(\Omega) } \le 
  C_4 h^{\min (k, s)} \\ 
\label{Maxwell_projection_prop3}
& \Vert \xi_{\Bb}\Vert_{H^{-1}(\Omega)} + \Vert \nabla \times \xi_{\Bb}\Vert_{H^{-1}(\Omega)} 
\leq C_{4} h^{k+1}.  
\end{align}
\end{subequations}
Here $C_4$ is a positive constant depending upon the physical parameters 
$S, R_m$, the domain $\Omega$ and the constant 
$K$ introduced in (\ref{assump_K}).
\end{theorem}

\begin{proof}
To prove the well-definedness of the modified Maxwell projection (\ref{Maxwell_projection}),  
we only consider the corresponding homogeneous system
\begin{align*}
a_{m} (\tilde{\Bb}_{h}, \Bc_{h}) + c_{1}(\tilde{\Bb}_{h}; \Bu, \Bc_{h})
 - (\nabla \tilde{r}_{h}, \Bc_{h})_{\Omega} = & 0, \\
 (\tilde{\Bb}_{h}, \nabla s_{h})_{\Omega} = & 0, 
\end{align*}
for any $(\Bc_{h}, s_{h}) \in \BC_{h}^{k} \times S_{h}^{k}$. 
Taking $\Bc_{h} = \tilde{\Bb}_{h}$ and $s_{h} = \tilde{r}_{h}$ leads to 
\begin{align*}
a_{m}(\tilde{\Bb}_h, \tilde{\Bb}_h) + c_{1} (\tilde{\Bb}_h; \Bu, \tilde{\Bb}_h) = 0.
\end{align*}
from which, we can see that 
\begin{align*}
R_{m}^{-1} \Vert \nabla \times \tilde{\Bb}_{h} \Vert_{L^{2}(\Omega)} ^{2} 
& \leq \Vert \tilde{\Bb}_{h}\Vert_{L^{3}(\Omega)} \Vert \Bu\Vert_{L^{6}(\Omega)} 
\Vert \nabla \times \tilde{\Bb}_{h} \Vert_{L^{2}(\Omega)} 
\\ 
&   \leq  \lambda_1 \lambda_2^* \Vert \nabla \Bu\Vert_{L^2(\Omega)} 
\Vert \nabla \times \tilde{\Bb}_{h} \Vert_{L^{2}(\Omega)}^2  \, . 
\end{align*}
This further shows that 
$$ 
(1 - \widehat N_2 \eta(\lambda^*_0)) \Vert \nabla \times \tilde{\Bb}_{h} \Vert_{L^{2}(\Omega)}^2 
\leq 0 \, . 
$$ 
By noting the condition (\ref{cond-2}),  
we get  $\tilde{\Bb}_{h} = \boldsymbol{0}$ in $\Omega$. 
It is straightforward to verify $\tilde{r}_{h} = 0$ in $\Omega$. 
Thus the modified Maxwell projection (\ref{Maxwell_projection})
is well defined when the condition (\ref{cond-2}) holds. 

With the standard Maxwell projection (\ref{maxwell_project_standard}), 
we rewrite  the system (\ref{err_eq2_group1}) and (\ref{err_eq4_group1}) into  
\begin{align*}
 & a_{m} (\widehat{\Bb}_{h} - \tilde{\Bb}_{h}, \Bc_{h}) + c_{1}(\widehat{\Bb}_{h} - \tilde{\Bb}_{h}; \Bu, \Bc_{h})
 - (\nabla (\widehat{r}_{h} - \tilde{r}_{h}), \Bc_{h})_{\Omega} \\
& \qquad \qquad =  a_{m} (\widehat{\Bb}_{h} - \Bb, \Bc_{h}) 
 + c_{1}(\widehat{\Bb}_{h} - \Bb; \Bu, \Bc_{h})
 - (\nabla (\widehat{r}_{h} - r), \Bc_{h})_{\Omega} ,  \\
& (\widehat{\Bb}_{h} - \tilde{\Bb}_{h}, \nabla s_{h})_{\Omega} = 0,  
\end{align*}
for any $(\Bc_{h}, s_{h}) \in \BC_{h}^{k} \times S_{h}^{k}$.
By taking $(\Bc_{h}, s_{h}) = (\widehat{\Bb}_{h} - \tilde{\Bb}_{h}, \widehat{r}_{h} - \tilde{r}_{h})$ 
in above two equations and applying (\ref{maxwell_project_standard}), we have 
\begin{align*}
& a_{m}(\widehat{\Bb}_{h} - \tilde{\Bb}_{h}, \widehat{\Bb}_{h} - \tilde{\Bb}_{h}) 
+ c_{1}(\widehat{\Bb}_{h} - \tilde{\Bb}_{h}; \Bu, \widehat{\Bb}_{h} - \tilde{\Bb}_{h}) \\
= &  a_{m}(\widehat{\Bb}_{h} - \Bb, \widehat{\Bb}_{h} - \tilde{\Bb}_{h}) 
+ c_{1}(\widehat{\Bb}_{h} - \Bb; \Bu, \widehat{\Bb}_{h} - \tilde{\Bb}_{h}) 
- (\nabla (\widehat{r}_{h} - r), \widehat{\Bb}_{h} - \tilde{\Bb}_{h})_{\Omega} \\
= & c_{1}(\widehat{\Bb}_{h} - \Bb; \Bu, \widehat{\Bb}_{h} - \tilde{\Bb}_{h}) ,
\end{align*} 
where we have noted that $\widehat{\Bb}_{h} - \tilde{\Bb}_{h} \in \BX_{h}$ (see (\ref{X_h_space})). 
By Lemma~\ref{lemma_poincare_embedding} and Lemma~\ref{lemma_terms_ineqs}, we further see that
\begin{align*}
& \Vert \widehat{\Bb}_{h} - \tilde{\Bb}_{h}  \Vert_{H(\text{curl}, \Omega)} 
\leq C \Vert \Bu \Vert_{L^{\infty}(\Omega)}   \Vert \widehat{\Bb}_{h} - \Bb \Vert_{L^{2}(\Omega)} \\ 
\leq & C h^{\min (k, s)}  \Vert \Bu \Vert_{L^{\infty}(\Omega)} 
\left(\Vert \Bb\Vert_{H^{s}(\Omega)} + \Vert \nabla \times \Bb \Vert_{H^{s}(\Omega)} 
+ \Vert r\Vert_{H^{s+1}(\Omega)} \right) 
\end{align*}
where we have used the approximation error (\ref{M-projection-error}). 
Therefore, 
\begin{align}
\label{maxwell_proj_approx1} 
& \Vert \Bb - \tilde{\Bb}_{h}  \Vert_{H(\text{curl}, \Omega)} \\
\nonumber
\leq & C h^{\min (k, s)} \left( \Vert \Bu \Vert_{L^{\infty}(\Omega)} + 1 \right)
\left(\Vert \Bb\Vert_{H^{s}(\Omega)} + \Vert \nabla \times \Bb \Vert_{H^{s}(\Omega)} 
+ \Vert r\Vert_{H^{s+1}(\Omega)} \right).
\end{align}
Since 
$\nabla \widehat r_h, \nabla \tilde r_h \in \BC_h^k$, taking $\Bc_h = 
\nabla \widehat r_h - \nabla \tilde r_h$ and noting 
\begin{align*}
a_{m} (\widehat{\Bb}_{h} - \tilde{\Bb}_{h},  \nabla (\widehat{r}_{h} - \tilde{r}_{h}) ) 
 & =   c_{1}(\widehat{\Bb}_{h} - \tilde{\Bb}_{h}; \Bu, \nabla (\widehat{r}_{h} - \tilde{r}_{h})) 
\\ 
& = c_{1}(\widehat{\Bb}_{h} - \Bb; \Bu, \nabla (\widehat{r}_{h} - \tilde{r}_{h})) = 0, 
\end{align*}
we get
\begin{align}
\label{maxwell_proj_approx2} 
\widehat{r}_{h} = \tilde{r}_{h}.
\end{align}
(\ref{Maxwell_projection_prop1}) follows (\ref{maxwell_proj_approx1}, \ref{maxwell_proj_approx2})
and the approximation property of the standard Maxwell projection (\ref{M-projection-error}).

Moreover, let $\Pi_{h}^{\text{curl}}$ be the projection $P_{1}$ onto $\BC_{h}^{k}$ in \cite[Section~$5$]{Christiansen11}. 
Then by \cite[Proposition~$5.65$]{Christiansen11}, we have 
\begin{align*}
\Vert \Bb -  \Pi_{h}^{\text{curl}}\Bb \Vert_{L^{3}(\Omega)} \leq & Ch^{\min (k, s)} \Vert \Bb\Vert_{W^{s, 3}(\Omega)}, \\
\Vert \Bb - \Pi_{h}^{\text{curl}}\Bb\Vert_{H(\text{curl}, \Omega)} \leq & C h^{\min (k, s)}\left(\Vert \Bb\Vert_{H^{s}(\Omega)} 
+ \Vert \nabla \times \Bb \Vert_{H^{s}(\Omega)} \right).
\end{align*}
We define $\sigma_{h} \in S_{h}^{k}$ by
\begin{align*}
(\nabla \sigma_{h}, \nabla s_{h})_{\Omega} = (\Pi_{h}^{\text{curl}}\Bb, \nabla s_{h})_{\Omega}, \qquad 
\forall s_{h} \in S_{h}^{k}.
\end{align*}
Since $ \Pi_{h}^{\text{curl}}\Bb - \nabla \sigma_{h}\in \BX_{h}$, by Lemma~\ref{lemma_discrete_curl_embedding} 
 and (\ref{Maxwell_projection_prop1}), 
\begin{align*}
& \Vert (\ \Pi_{h}^{\text{curl}}\Bb - \nabla\sigma_{h}) - \tilde{\Bb}_{h} \Vert_{L^{3}(\Omega)} 
\leq C \Vert  \Pi_{h}^{\text{curl}}\Bb - \tilde{\Bb}_{h}\Vert_{H(\text{curl}, \Omega)} \\
\leq & C h^{\min (k, s)} \left( \Vert \Bu\Vert_{H^{2}(\Omega)} + 1\right) \left( \Vert \Bb\Vert_{H^{s}(\Omega)} 
+ \Vert \nabla \times \Bb\Vert_{H^{s}(\Omega)} + \Vert r\Vert_{H^{s+1}(\Omega)} \right).
\end{align*}
We define by $\sigma$ the solution of 
\begin{align*}
\Delta \sigma = \nabla \cdot (\Bb - \Pi_{h}^{\text{curl}}\Bb) \text{ in } \Omega, \qquad 
\sigma = 0 \text{ on } \partial \Omega.
\end{align*}
Following \cite[Theorem~$0.5$]{JerisonKenig1995} (see also \cite[Corollary~$3.10$]{Dauge92} 
and \cite[Remark~$3.11$]{Dauge92} ),
\begin{align*}
\Vert \sigma\Vert_{W^{1, 3}(\Omega)} \leq C \Vert \Bb - \Pi_{h}^{\text{curl}}\Bb\Vert_{L^{3}(\Omega)} 
\leq C h^{\min (k, s)} \Vert \Bb\Vert_{W^{s,3}(\Omega)}.
\end{align*}
By the definition of $\sigma$ and noting the fact that $\nabla\cdot \Bb = 0$, we see that 
\begin{align*}
(\nabla \sigma_{h}, \nabla s_{h} )_{\Omega} = ( \Pi_{h}^{\text{curl}}\Bb - \Bb, \nabla s_{h})_{\Omega}
= (\nabla \sigma, \nabla s_{h})_{\Omega}, \qquad 
\forall s_{h} \in S_{h}^{k}.  
\end{align*}
By \cite[Theorem~$2$]{Guzman09} with the assumption $\Omega$ being convex and standard interpolation 
argument for bounded linear operator, we have 
\begin{align*}
\Vert \sigma_{h}  \Vert_{W^{1,3}(\Omega)} \leq C \Vert \sigma\Vert_{W^{1,3}(\Omega)} 
\leq C h^{\min (k, s)} \Vert \Bb\Vert_{W^{s,3}(\Omega)}, 
\end{align*}
 and 
\begin{align}
\label{maxwell_proj_approx_L3}
& \Vert \Bb - \tilde{\Bb}_{h}\Vert_{L^{3}(\Omega)} \\
\nonumber
\leq & \Vert  \Bb - \Pi_{h}^{\text{curl}} \Bb\Vert_{L^{3}(\Omega)} + \Vert \nabla \sigma_{h} \Vert_{L^{3}(\Omega)} 
+ \Vert (\Pi_{h}^{\text{curl}}\Bb - \nabla \sigma_{h}) - \tilde{\Bb}_{h} \Vert_{L^{3}(\Omega)} \\
\nonumber 
\leq & C h^{\min (k, s)} \left( \Vert \Bu\Vert_{H^{2}(\Omega)} + 1\right) \left(   \Vert \Bb\Vert_{H^{s}(\Omega)} 
+ \Vert \nabla \times \Bb \Vert_{H^{s}(\Omega)} + \Vert \Bb\Vert_{W^{s,3}(\Omega)} + \Vert r\Vert_{H^{s+1}(\Omega)} \right).
\end{align}
(\ref{Maxwell_projection_prop2}) follows immediately.

We notice that 
\begin{align*}
 \Vert \Bb - \tilde{\Bb}_{h}\Vert_{H^{-1}(\Omega)} 
 = \sup_{\Btheta \in [H_{0}^{1}(\Omega)]^d} \dfrac{(\Bb - \tilde{\Bb}_{h}, \Btheta)_{\Omega} }{\Vert 
 \Btheta\Vert_{H^{1}(\Omega)}}. 
 \end{align*}
 Let $(\Bz, \phi) \in H_{0}(\text{curl}, \Omega) \times H_{0}^{1}(\Omega)$ 
be the solution of the system (\ref{dual_problem_eq1})-(\ref{dual_problem_eq2}). 
Since the condition (\ref{cond-2}) implies the condition (\ref{cond-1}) and  the assumption 
$\Bu \in [H^{2}(\Omega)]^d$ , 
by Theorem~\ref{thm_dual_problem}, the system \refe{dual_problem_eqs} 
is well-posed and  
\begin{align}
\label{dual_regularity_in_proof}
\Vert \Bz\Vert_{H^{1}(\Omega)} + \Vert \nabla \times \Bz \Vert_{H^{1}(\Omega)} + \Vert \phi \Vert_{H^{2}(\Omega)} 
\leq C_{3} \Vert \Btheta \Vert_{H^{1}(\Omega)},
\end{align} 
where the constant $C_{3}$ is introduced in Theorem~\ref{thm_dual_problem}.
Then we have 
\begin{align}
\label{replacement1}
 & (\Bb - \tilde{\Bb}_{h}, \Btheta)_{\Omega} \\
\nonumber 
= & S R_{m}^{-1} (\nabla \times (\Bb - \tilde{\Bb}_{h}), \nabla \times z)_{\Omega} 
 + S (\Bu \times (\nabla \times \Bz), \Bb - \tilde{\Bb}_{h})_{\Omega} - (\Bb - \tilde{\Bb}_{h}, \nabla \phi)_{\Omega} \\ 
\nonumber  
 = & a_{m} (\Bb - \tilde{\Bb}_{h}, \Bz) + c_{1}(\Bb - \tilde{\Bb}_{h}; \Bu, z) - (\Bb - \tilde{\Bb}_{h}, \nabla \phi)_{\Omega} \\ 
\nonumber 
 = & a_{m} (\Bb - \tilde{\Bb}_{h}, \Bz - \Bz_{h}) + c_{1}(\Bb - \tilde{\Bb}_{h}; \Bu, \Bz - \Bz_{h}) 
 + (\nabla (r - \tilde{r}_{h}), \Bz - \Bz_{h})_{\Omega}
\\ 
\nonumber 
&  - (\Bb - \tilde{\Bb}_{h}, \nabla (\phi - \phi_{h}))_{\Omega},  
\end{align}
for any $(\Bz_{h}, \phi_{h}) \in \BC_{h}^{k} \times S_{h}^{k}$. The last equality  
follows the definition of the modified 
Maxwell projection (\ref{Maxwell_projection}) and the fact that $\nabla \cdot \Bz = 0$. 
By Lemma~\ref{lemma_terms_ineqs}, we further have  
\begin{align}
\label{replacement2}
  (\Bb - \tilde{\Bb}_{h}, \Btheta)_{\Omega} 
 \leq & C \big(  (1 + \Vert \Bu\Vert_{L^{\infty}(\Omega})\Vert \Bb - \tilde{\Bb}_{h} \Vert_{H(\text{curl}, \Omega)} 
 \Vert \nabla \times (\Bz - \Bz_{h})  \Vert_{L^{2}(\Omega)}   \\ 
\nonumber 
& \quad  + \Vert \nabla (r - \tilde{r}_{h}) \Vert_{L^{2}(\Omega)} \Vert \Bz - \Bz_{h}\Vert_{L^{2}(\Omega)}  
+ \Vert \Bb - \tilde{\Bb}_{h}\Vert_{L^{2}(\Omega)} \Vert \nabla (\phi - \phi_{h}) \Vert_{L^{2}(\Omega)}  \big).
\end{align}
We choose $\Bz_{h}$ and $\phi_{h}$ to be the best approximations to $z$ and $\phi$
in $\BC_{h}^{k}$ and $S_{h}^{k}$ for $H(\text{curl})$-norm and $H^{1}$-norm, respectively. 
By (\ref{dual_regularity_in_proof}), 
\begin{align}
\label{maxwell_proj_approx3}
 \Vert  \Bb - \tilde{\Bb}_{h} \Vert_{H^{-1}(\Omega)} 
\leq & C C_{3} h \left( \Vert \Bb - \tilde{\Bb}_{h}\Vert_{H(\text{curl}, \Omega)}  
+ \Vert \nabla (r - \tilde{r}_{h}) \Vert_{L^{2}(\Omega)} \right) \\ 
\nonumber
\leq & C C_{3} h^{k+1} \left(\Vert \Bb\Vert_{H^{k}(\Omega)} + \Vert \nabla \times \Bb \Vert_{H^{k}(\Omega)} 
+ \Vert r\Vert_{H^{k+1}(\Omega)} \right).
\end{align}

On the other hand, it is easy to see 
\begin{align}
\label{replacement3}
\Vert \nabla \times (\Bb - \tilde{\Bb}_{h}) \Vert_{H^{-1}(\Omega)} 
= \sup_{\boldsymbol{\eta} \in [H_{0}^{1}(\Omega)]^d} \dfrac{(\nabla \times (\Bb - \tilde{\Bb}_{h}),
 \boldsymbol{\eta})_{\Omega}}{\Vert \boldsymbol{\eta}\Vert_{H^{1}(\Omega)}} 
 =  \sup_{\boldsymbol{\eta} \in [H_{0}^{1}(\Omega)]^d} \dfrac{( \Bb - \tilde{\Bb}_{h},
 \nabla \times \boldsymbol{\eta})_{\Omega}}{\Vert \boldsymbol{\eta}\Vert_{H^{1}(\Omega)}}. 
\end{align}
Then for any $\boldsymbol{\eta} \in [H_{0}^{1}(\Omega)]^{d}$, the same argument obtaining (\ref{replacement1}) implies   
\begin{align*}
(\Bb - \tilde{\Bb}_{h}, \nabla \times \boldsymbol{\eta})_{\Omega} 
 = & a_{m} (\Bb - \tilde{\Bb}_{h}, \Bz - \Bz_{h}) + c_{1}(\Bb - \tilde{\Bb}_{h}; \Bu, \Bz - \Bz_{h}) 
 + (\nabla (r - \tilde{r}_{h}), \Bz - \Bz_{h})_{\Omega} \\ 
&  - (\Bb - \tilde{\Bb}_{h}, \nabla (\phi - \phi_{h}))_{\Omega},  
\qquad \forall (\Bz_{h}, \phi_{h}) \in \BC_{h}^{k} \times S_{h}^{k}.
\end{align*}
Here  $(\Bz, \phi) \in H_{0}(\text{curl}, \Omega) \times H_{0}^{1}(\Omega)$ 
is the solution of the system (\ref{dual_problem_eq1})-(\ref{dual_problem_eq1}) 
with $\Btheta = \nabla \times \boldsymbol{\eta}$. 
By Lemma~\ref{lemma_terms_ineqs}, we have 
\begin{align*}
  (\Bb - \tilde{\Bb}_{h}, \nabla \times \boldsymbol{\eta})_{\Omega} 
 \leq & C \big(  (1 + \Vert \Bu\Vert_{L^{\infty}(\Omega})\Vert \Bb - \tilde{\Bb}_{h} \Vert_{H(\text{curl}, \Omega)} 
 \Vert \nabla \times (\Bz - \Bz_{h})  \Vert_{L^{2}(\Omega)}   \\ 
& \quad  + \Vert \nabla (r - \tilde{r}_{h}) \Vert_{L^{2}(\Omega)} \Vert \Bz - \Bz_{h}\Vert_{L^{2}(\Omega)}  
+ \Vert \Bb - \tilde{\Bb}_{h}\Vert_{L^{2}(\Omega)} \Vert \nabla (\phi - \phi_{h}) \Vert_{L^{2}(\Omega)}  \big).
\end{align*}
Notice that $\Vert \Btheta \Vert_{H(\text{div}, \Omega)} = \Vert \nabla \times \boldsymbol{\eta}\Vert_{L^{2}(\Omega)} 
\leq \Vert \boldsymbol{\eta}\Vert_{H^{1}(\Omega)}$. By Theorem~\ref{thm_dual_problem}, we have 
\begin{align}
\label{dual_regularity_in_proof_2}
\Vert \Bz\Vert_{H^{1}(\Omega)} + \Vert \nabla \times \Bz \Vert_{H^{1}(\Omega)} + \Vert \phi \Vert_{H^{2}(\Omega)} 
\leq C_{3} \Vert \Btheta \Vert_{H(\text{div},\Omega)} 
\leq C_{3}\Vert \boldsymbol{\eta}\Vert_{H^{1}(\Omega)},
\end{align} 
where the constant $C_{3}$ is introduced in Theorem~\ref{thm_dual_problem}.
We choose $\Bz_{h}$ and $\phi_{h}$ to be the best approximations to $z$ and $\phi$
in $\BC_{h}^{k}$ and $S_{h}^{k}$ for $H(\text{curl})$-norm and $H^{1}$-norm, respectively. 
By (\ref{dual_regularity_in_proof_2}), 
\begin{align*}
 (\Bb - \tilde{\Bb}_{h}, \nabla \times \boldsymbol{\eta})_{\Omega} 
\leq & C C_{3} h \left( \Vert \Bb - \tilde{\Bb}_{h}\Vert_{H(\text{curl}, \Omega)}  
+ \Vert \nabla (r - \tilde{r}_{h}) \Vert_{L^{2}(\Omega)} \right) \Vert \boldsymbol{\eta} \Vert_{H^{1}(\Omega)} \\ 
\leq & C C_{3} h^{k+1} \left(\Vert \Bb\Vert_{H^{k}(\Omega)} + \Vert \nabla \times \Bb \Vert_{H^{k}(\Omega)} 
+ \Vert r\Vert_{H^{k+1}(\Omega)} \right) \Vert \boldsymbol{\eta} \Vert_{H^{1}(\Omega)}.
\end{align*}
(\ref{replacement3}) and the last inequality implies 
\begin{align}
\label{maxwell_proj_approx4}
& \Vert \nabla \times ( \Bb - \tilde{\Bb}_{h} ) \Vert_{H^{-1}(\Omega)} \\ 
\nonumber
\leq & C C_{3} h \left( \Vert \Bb - \tilde{\Bb}_{h}\Vert_{H(\text{curl}, \Omega)}  
+ \Vert \nabla (r - \tilde{r}_{h}) \Vert_{L^{2}(\Omega)} \right) \\ 
\nonumber
\leq & C C_{3} h^{k+1} \left(\Vert \Bb\Vert_{H^{k}(\Omega)} + \Vert \nabla \times \Bb \Vert_{H^{k}(\Omega)} 
+ \Vert r\Vert_{H^{k+1}(\Omega)} \right) 
\end{align}
by the same argument with (\ref{dual_regularity_in_proof}) in last paragraph. 
Thus (\ref{Maxwell_projection_prop3}) holds by (\ref{maxwell_proj_approx3}) and (\ref{maxwell_proj_approx4}). 

Now we conclude that the proof is complete.
\end{proof}

\begin{remark} 
If the condition (\ref{cond-2}) holds, we claim that 
\begin{align}
\label{Maxwell_projection_L_infinity_bound}
\Vert \tilde{\Bb}_{h}\Vert_{L^{\infty}(\Omega)} \leq C.
\end{align}
Here $\tilde{\Bb}_{h}$ is defined in (\ref{Maxwell_projection}). 
  
By (\ref{Maxwell_projection_prop2}) and the condition (\ref{cond-2}), we have that 
\begin{align*}
\Vert \Bb - \tilde{\Bb}_{h} \Vert_{L^{3}(\Omega)} \leq Ch.
\end{align*} 
We denote by $\Bb_{h}^{0}$ the standard $L^{2}$-orthogonal projection of $\Bb$ onto $[P_{0}(\Ct_{h})]^{d}$. 
By the condition (\ref{cond-2}) again, we have that 
\begin{align*}
& \Vert \Bb_{h}^{0} \Vert \leq C \Vert \Bb\Vert_{L^{\infty}(\Omega)} 
\leq C \Vert \Bb \Vert_{W^{1, d^{+}}(\Omega)} \leq C,\\
& \Vert \Bb - \Bb_{h}^{0} \Vert_{L^{3}(\Omega)} \leq Ch.
\end{align*}
Then by discrete inverse inequality, we have that 
\begin{align*}
&\Vert \tilde{\Bb}_{h} \Vert_{L^{\infty}(\Omega)} 
\leq \Vert \tilde{\Bb}_{h} - \Bb_{h}^{0} \Vert_{L^{\infty}(\Omega)} + \Vert \Bb_{h}^{0}\Vert_{L^{\infty}(\Omega)} \\
\leq & C h^{-1}\Vert \tilde{\Bb}_{h} - \Bb_{h}^{0} \Vert_{L^{3}(\Omega)} 
+ \Vert \Bb_{h}^{0}\Vert_{L^{\infty}(\Omega)} \leq C .
\end{align*}
Thus (\ref{Maxwell_projection_L_infinity_bound}) holds. 

In fact, by (\ref{Maxwell_projection_prop2},\ref{mfem_conv}) and the same argument above, we have that 
\begin{align}
\label{numerical_b_L_infinite_bound}
\Vert \Bb_{h}\Vert_{L^{\infty}(\Omega)} \leq C,
\end{align}
if the condition (\ref{cond-2}) holds.
\end{remark}

\vskip0.1in 

\subsection{Proof of Theorem~\ref{thm_main_result}}
By Lemma \ref{thm_mfem_wellposedness}, the mixed finite element system (\ref{MFEM_mhd}) admits a unique solution and the boundedness \refe{bound} holds. 

From (\ref{mhd_eqs}) and (\ref{MFEM_mhd}), we can see that the error functions satisfy the following equations 
\begin{subequations}
\label{err_eqs_group1}
\begin{align}
\label{err_eq1_group1}
& a_{s}(\Bu - \Bu_{h}, \Bv_{h}) =  -\left( c_{0}(\Bu; \Bu, \Bv_{h}) - c_{0}(\Bu_{h}; \Bu_{h}, \Bv_{h}) \right) 
 + \left( c_{1}(\Bb; \Bv_{h}, \Bb) - c_{1}(\Bb_{h}; \Bv_{h}, \Bb_{h})  \right)  \\ 
\nonumber 
& \qquad \qquad \qquad \qquad \qquad  - (p - p_{h}, \nabla\cdot \Bv_{h})_{\Omega}, \\
\label{err_eq2_group1}
& a_{m}(\Bb - \Bb_{h}, \Bc_{h}) = - \left( c_{1}(\Bb; \Bu, \Bc_{h}) - c_{1}(\Bb_{h}; \Bu_{h}, \Bc_{h}) \right) 
+ (\nabla (r - r_{h}), \Bc_{h})_{\Omega}, \\
\label{err_eq3_group1}
& (\nabla\cdot (\Bu - \Bu_{h}), q_{h})_{\Omega} = 0, \\
\label{err_eq4_group1}
& (\Bb - \Bb_{h}, \nabla s_{h})_{\Omega} = 0,
\end{align}
\end{subequations}
for any $(\Bv_{h}, \Bc_{h}, q_{h}, s_{h}) \in \BV_{h}^{l}\times \BC_{h}^{k}\times Q_{h}^{l}\times S_{h}^{k}$.

With the splitting (\ref{err_splits}), by taking  $\Bv_{h} = e_{\Bu}$, $\Bc_{h} = e_{\Bb}$, 
$q_{h} = e_{p}$ and $s_{h} = e_{r}$, the error equations (\ref{err_eqs_group1}) reduce to 
\begin{subequations}
\label{err_eqs_group2}
\begin{align}
\label{err_eq1_group2}
& a_{s}(e_{\Bu}, e_{\Bu}) + (e_{p}, \nabla\cdot e_{\Bu})_{\Omega}
=  -\left( c_{0}(\Bu; \Bu, e_{\Bu}) - c_{0}(\Bu_{h}; \Bu_{h}, e_{\Bu}) \right) \\
\nonumber
& \qquad \qquad \qquad \qquad \qquad  \qquad \qquad
+ \left( c_{1}(\Bb; e_{\Bu}, \Bb) - c_{1}(\Bb_{h}; e_{\Bu} ,  \Bb_{h})  \right),  \\ 
\label{err_eq2_group2}
& a_{m}(e_{\Bb}, e_{\Bu}) - (\nabla e_{r}, e_{\Bb})_{\Omega} = 
- \left( c_{1}(\Bb; \Bu, e_{\Bb}) - c_{1}(\Bb_{h}; \Bu_{h}, e_{\Bb}) 
- c_{1}(\xi_{\Bb}; \Bu, e_{\Bb}) \right), \\
\label{err_eq3_group2}
& (\nabla\cdot e_{\Bu} , e_{p})_{\Omega} = 0, \\
\label{err_eq4_group2}
& (e_{\Bb}, \nabla e_{r})_{\Omega} = 0,
\end{align}
\end{subequations}
where we have noted the definition of these two projections (\ref{Stokes_projection}) and (\ref{Maxwell_projection}). 
Notice that if we used the standard Maxwell projection (\ref{maxwell_project_standard}), then the term 
$c_{1}(\xi_{\Bb}; \Bu, e_{\Bb})$ would not appear in the error equation (\ref{err_eq2_group2}).

Summing up the first two equations in (\ref{err_eqs_group2}) leads to  
\begin{align}
& a_{s}(e_{\Bu}, e_{\Bu})  + a_{m}(e_{\Bb}, e_{\Bb}) 
\nn  \\
= & -\left(  c_{0}(\Bu; \Bu, e_{\Bu}) - c_{0}(\Bu_{h}; \Bu_{h}, e_{\Bu}) \right) \nn \\
& \quad + \left[ \left( c_{1}(\Bb; e_{\Bu}, \Bb) - c_{1}(\Bb_{h}; e_{\Bu} ,  \Bb_{h})  \right)
- \left(  c_{1}(\Bb; \Bu, e_{\Bb}) - c_{1}(\Bb_{h}; \Bu_{h}, e_{\Bb}) 
- c_{1}(\xi_{\Bb}; \Bu, e_{\Bb}) \right)  \right] 
\nn  \\
:= &  I_{\Bu} + I_{\Bb}.
\label{as+am} 
\end{align}
By the skew-symmetry of the operator $c_{0}$, Lemma~\ref{lemma_terms_ineqs} 
and Theorem~\ref{thm_mfem_wellposedness}, 
\begin{align}
\label{Ia}
I_{\Bu} = & - \left( c_{0}(\xi_{\Bu}; \Bu, e_{\Bu}) + c_{0}(e_{\Bu}; \Bu, e_{\Bu}) + c_{0}(\Bu_{h}; \xi_{\Bu}, e_{\Bu})\right) \\
\nn 
\leq & \lambda_{1} \lambda_{1}^{*} 
\Vert \nabla \Bu\Vert_{L^2(\Omega)} \Vert \nabla e_{\Bu}\Vert_{L^2(\Omega)}^{2} 
+ 
C \left(  \Vert \Bu\Vert_{W^{1,d^+}(\Omega)} +  \Vert \Bu_{h}\Vert_{W^{1,d+}(\Omega)} \right) 
\Vert \xi_{\Bu}\Vert_{L^2(\Omega)} \Vert e_{\Bu}\Vert_{H^{1}(\Omega)} \\
\nn
\leq & \left(\lambda_{1} \lambda_{1}^{*}  \Vert \nabla \Bu\Vert_{L^2(\Omega)} + \epsilon \right)
\Vert \nabla e_{\Bu}\Vert_{L^2(\Omega)}^{2}  
+ 
C\epsilon^{-1}  \left(  \Vert \Bu\Vert_{W^{1,d^+} (\Omega)} +  \Vert \Bu_{h}\Vert_{W^{1,d^+}(\Omega)} \right)^{2} 
 \Vert \xi_{\Bu}\Vert_{L^2(\Omega)}^2  
  \\ 
\nn
\leq &  \left(\lambda_{1} \lambda_{1}^{*} \Vert \nabla  \Bu\Vert_{L^2(\Omega)} + \epsilon \right)
\Vert \nabla e_{\Bu}\Vert_{L^2(\Omega)}^{2}  
+ C_K \epsilon^{-1} h^{2l+2} 
\end{align} 
we have used \refe{bound} and noted 
$\| \cdot \|_{L^{\infty}(\Omega)} \le C \| \cdot \|_{W^{1,d^+}(\Omega)}$. 
We recall that $d^{+}>d$ is a constant introduced in (\ref{assump_K}). 

On the other hand,  by rearranging terms in $I_{\Bb}$, we have 
\begin{align}
I_{\Bb} =  & \left( c_{1}(\Bb; e_{\Bu}, \Bb) - c_{1}(\Bb_{h}; e_{\Bu} ,  \Bb_{h})  \right)
- \left(  c_{1}(\Bb; \Bu, e_{\Bb}) - c_{1}(\Bb_{h}; \Bu_{h}, e_{\Bb}) - c_{1}(\xi_{\Bb}; \Bu, e_{\Bb}) \right) 
\nn  \\
= & c_{1}(\Bb - \Bb_{h}; e_{\Bu}, \Bb) + c_{1}(\Bb_{h}; e_{\Bu}, \Bb - \Bb_{h}) \nn \\
& \qquad  - c_{1}(\Bb - \Bb_{h}; \Bu, e_{\Bb}) - c_{1}(\Bb_{h}; \Bu - \Bu_{h}, e_{\Bb}) + c_{1}(\xi_{\Bb}; \Bu, e_{\Bb}) 
\nn \\  
= & c_{1}(\Bb - \Bb_{h}; e_{\Bu}, b) + c_{1}(\Bb_{h}; e_{\Bu}, \xi_{\Bb}) 
\nn  \\
& \qquad - c_{1}(\Bb - \Bb_{h}; \Bu, e_{\Bb}) - c_{1}(\Bb_{h}; \xi_{\Bu}, e_{\Bb}) 
+ c_{1}(\xi_{\Bb}; \Bu, e_{\Bb}).  
\label{I_b_init}
\end{align}
Notice that 
\begin{align}
& - c_{1}(\Bb - \Bb_{h}; \Bu, e_{\Bb}) + c_{1}(\xi_{\Bb}; \Bu, e_{\Bb}) \nn \\
= & - c_{1}(\xi_{\Bb}; \Bu, e_{\Bb}) - c_{1}(e_{\Bb}; \Bu, e_{\Bb}) + c_{1}(\xi_{\Bb}; \Bu, e_{\Bb}) 
= - c_{1}(e_{\Bb}; \Bu, e_{\Bb}).
\label{key_identity}
\end{align}
Then by (\ref{I_b_init}) and (\ref{key_identity}), we have   
\begin{align}
I_{\Bb} = & c_{1}(\Bb - \Bb_{h}; e_{\Bu}, \Bb) + c_{1}(\Bb_{h}; e_{\Bu}, \xi_{\Bb}) 
- c_{1}(e_{\Bb}; \Bu, e_{\Bb}) - c_{1}(\Bb_{h}; \xi_{\Bu}, e_{\Bb}) 
\nn \\
= &  c_{1}(\xi_{\Bb}; e_{\Bu}, \Bb)  + c_{1}(\Bb; e_{\Bu}, \xi_{\Bb}) 
+ \big( c_{1}(e_{\Bb}; e_{\Bu}, \Bb)  -c_{1}(e_{\Bb}; e_{\Bu}, \xi_{\Bb}) - c_{1}(\xi_{\Bb}; e_{\Bu}, \xi_{\Bb}) 
 \nn \\ 
& \qquad - c_{1}(e_{\Bb}; \Bu, e_{\Bb}) \big )  - c_{1}(\Bb_{h}; \xi_{\Bu}, e_{\Bb}) . 
\label{I_b} 
\end{align}
We estimate these terms in the right-hand side of the above equation below. 
By the definition of the operator $c_1$, the sum of the first two terms 
in the right-hand side above can be rewritten by 
\begin{align*}
 c_{1}(\xi_{\Bb}; e_{\Bu}, \Bb)  + c_{1}(\Bb; e_{\Bu}, \xi_{\Bb}) 
=  S (\xi_{\Bb},  e_{\Bu} \times (\nabla \times \Bb))_{\Omega} - S (\nabla \times \xi_{\Bb}, e_{\Bu} \times \Bb)_{\Omega}
\end{align*}
which, by Theorem~\ref{thm_Maxwell_projection}  and the fact that $e_{\Bu} \in [H_{0}^{1}(\Omega)]^{3}$, is bounded by  
\begin{align}
& c_{1}(\xi_{\Bb}; e_{\Bu}, \Bb)  + c_{1}(\Bb; e_{\Bu}, \xi_{\Bb}) 
\nn \\
\leq & S \left( \Vert \xi_{\Bb}\Vert_{H^{-1}(\Omega)} \Vert  e_{\Bu} \times (\nabla \times \Bb)\Vert_{H^{1}(\Omega)} 
+ \Vert \nabla \times \xi_{\Bb} \Vert_{H^{-1}(\Omega)} \Vert  e_{\Bu} \times \Bb \Vert_{H^{1}(\Omega)} \right)  
\nn \\
\leq & C \Vert \xi_{\Bb}\Vert_{H^{-1}(\Omega)} \left( \Vert e_{\Bu}\Vert_{L^{6}(\Omega)} 
\Vert \nabla \times \Bb \Vert_{W^{1,3}(\Omega)}  + \Vert \nabla e_{\Bu} \Vert_{L^{2}(\Omega)} 
 \Vert \nabla \times \Bb \Vert_{L^{\infty}(\Omega)}  \right) 
 \nn \\
& \qquad + C \Vert \nabla \times \xi_{\Bb} \Vert_{H^{-1}(\Omega)} 
\left( \Vert e_{\Bu}\Vert_{L^{6}(\Omega)} 
\Vert \Bb \Vert_{W^{1,3}(\Omega)}  + \Vert \nabla e_{\Bu} \Vert_{L^{2}(\Omega)} 
 \Vert \Bb \Vert_{L^{\infty}(\Omega)} \right)
 \nn \\ 
\leq & C \Vert \xi_{\Bb}\Vert_{H^{-1}(\Omega)} \Vert e_{\Bu}\Vert_{H^{1}(\Omega)} 
 \left( \Vert \nabla \times \Bb \Vert_{W^{1,3}(\Omega)}  
 +  \Vert \nabla \times \Bb \Vert_{L^{\infty}(\Omega)}  \right) 
 \nn  \\
& \qquad + C \Vert \nabla \times \xi_{\Bb} \Vert_{H^{-1}(\Omega)} \Vert e_{\Bu}\Vert_{H^{1}(\Omega)} 
\left( \Vert \Bb \Vert_{W^{1,3}(\Omega)}  +  \Vert \Bb \Vert_{L^{\infty}(\Omega)} \right) 
\nn \\ 
\leq & C \left( \Vert \Bb\Vert_{W^{1, d^{+}}(\Omega)} +  \Vert \nabla \times \Bb\Vert_{W^{1,d^{+}}(\Omega)} \right) 
\Vert e_{\Bu} \Vert_{H^{1}(\Omega)} \left( \Vert \xi_{\Bb}\Vert_{H^{-1}(\Omega)} 
+ \Vert \nabla \times \xi_{\Bb}\Vert_{H^{-1}(\Omega)} \right) 
\nn \\
\leq & C_K h^{k+1} \Vert e_{\Bu}\Vert_{H^{1}(\Omega)} \, . 
\label{negative_norm_key}
\end{align}

We notice $e_{\Bb}, \Bb_{h} \in \BX_{h}$. 
By Theorem~\ref{thm_Maxwell_projection}, (\ref{Maxwell_projection_L_infinity_bound})
and the definition of $c_{1}$, the last term in (\ref{I_b}) is bounded by 
\begin{align} 
|c_{1}(\Bb_{h}; \xi_{\Bu}, e_{\Bb})| 
 & =  |c_{1}(\tilde \Bb_{h}; \xi_{\Bu}, e_{\Bb}) - 
c_{1}(e_\Bb; \xi_{\Bu}, e_{\Bb}) | 
\nn \\ 
& \le 
C ( \| \nabla \times e_{\Bb} \|_{L^2(\Omega)} \Vert \xi_{\Bu}\Vert_{L^6(\Omega)} 
+ \Vert \tilde \Bb_h \Vert_{L^{\infty}(\Omega)} \Vert \xi_{\Bu}\Vert_{L^2(\Omega)} )
\Vert e_{\Bb}\Vert_{H(\text{curl}, \Omega)}  
\nn \\ 
& 
\le \epsilon \Vert e_{\Bb}\Vert_{H(\text{curl}, \Omega)}^2 
+ C_K h^{2l+2} 
\nn
\end{align} 
where we have noted $\Vert \xi_{\Bu}\Vert_{L^6(\Omega)}  \le C \| \xi_{\Bu} \|_{H^1(\Omega)} 
\le C_K h^l \le \epsilon$ when 
$h \le h_0$ for some $h_0>0$.
Moreover, 
by Lemma~\ref{lemma_poincare_embedding}, 
Lemma~\ref{lemma_discrete_curl_embedding},  
the sum of the rest terms in the right-hand side of 
(\ref{I_b}) is bounded by 
\begin{align*}
c_{1}(e_{\Bb}; e_{\Bu}, \Bb) 
& -  c_{1}(e_{\Bb}; e_{\Bu}, \xi_{\Bb}) -    c_{1}(\xi_{\Bb}; e_{\Bu}, \xi_{\Bb})
- c_{1}(e_{\Bb}; \Bu, e_{\Bb}) 
\\
& \le  S \lambda_{1}\lambda_{2}^{*} \Vert \nabla \times \Bb \Vert_{L^2(\Omega)} 
\Vert \nabla \times e_{\Bb}\Vert_{L^2(\Omega)} 
\Vert \nabla e_{\Bu}\Vert_{L^2(\Omega)}   
\nn \\ 
& \quad + C \Vert e_{\Bb}\Vert_{H(\text{curl},\Omega)} 
\Vert e_{\Bu}\Vert_{H^{1}(\Omega)} \Vert \xi_{\Bb}\Vert_{H(\text{curl}, \Omega)} 
 \\
& \quad + C \Vert \xi_{\Bb}\Vert_{L^3(\Omega)} \Vert e_{\Bu}\Vert_{H^{1}(\Omega)} \Vert \xi_{\Bb}\Vert_{H(\text{curl}, \Omega)}
+ S \lambda_{1} \lambda_{2}^{*} \Vert \nabla \Bu\Vert_{L^2(\Omega)} \Vert \nabla \times e_{\Bb}\Vert_{L^2(\Omega)}^{2} 
 \\ 
& \le  \frac{1}{2} S^{1/2} \lambda_{1} \lambda_{2}^{*} \Vert \nabla \times \Bb \Vert_{L^{2}(\Omega)} 
\| ( e_\Bb, e_\Bu) \|^2  
+  C \Vert \xi_{\Bb}\Vert_{H(\text{curl}, \Omega)}
\| ( e_\Bb, e_\Bu) \|^2  
\nn \\
& \quad + \frac{\epsilon}{2}  \Vert e_{\Bu}\Vert_{H^{1}(\Omega)}^2 
+ C \epsilon^{-1} \Vert \xi_{\Bb}\Vert_{L^{3}(\Omega)}^{2}  \Vert \xi_{\Bb}\Vert_{H(\text{curl}, \Omega)}^{2}
+ S \lambda_{1} \lambda_{2}^{*} \Vert \nabla \Bu\Vert_{L^2(\Omega)} 
\| \nabla \times e_{\Bb} \|_{L^2(\Omega)}^2 
 \\
& \le  \left (\epsilon +C h + \frac{1}{2}S^{1/2} \lambda_{1} \lambda_{2}^{*} 
\Vert \nabla \times \Bb \Vert_{L^{2}(\Omega)}  \right  ) \| ( e_\Bb, e_\Bu) \|^2 
\nn \\ 
&  + S \lambda_{1}\lambda_{2}^{*} \Vert  \nabla \Bu\Vert_{L^2(\Omega)}
  \Vert \nabla \times e_\Bb \Vert_{L^2(\Omega)}^2 
  + C_K \epsilon^{-1} h^{2(k+1)} \, . 
\end{align*} 
The last inequality holds since $k \geq 1$. 
By combining the above inequalities, we get the estimate   
\begin{align}
\label{Ib}
I_\Bb & \leq  \left (2\epsilon +C h + \frac{1}{2}S^{1/2} \lambda_{1} \lambda_{2}^{*} 
\Vert \nabla \times \Bb \Vert_{L^{2}(\Omega)}  \right  ) \| ( e_\Bb, e_\Bu) \|^2 
 + S \lambda_{1}\lambda_{2}^{*} \Vert  \nabla \Bu\Vert_{L^2(\Omega)}
  \Vert \nabla \times e_\Bb \Vert_{L^2(\Omega)}^2 
  \nn \\ 
  & \qquad 
  +  \epsilon^{-1} C_K h^{2(k+1)}.
\end{align}

Substituting (\ref{Ia})-(\ref{Ib}) into (\ref{as+am}) gives 
\begin{align} 
a_s(e_\Bu, e_\Bu) + a_m(e_\Bb, e_\Bb)  
& \le  \left (\epsilon +C h + \frac{1}{2}S^{1/2} \lambda_{1} \lambda_{2}^{*} 
\Vert \nabla \times \Bb \Vert_{L^{2}(\Omega)}  \right  ) \| ( e_\Bb, e_\Bu) \|^2 
\nn \\ 
& \quad  +  \max \{  \lambda_{1}\lambda_{2}^{*}, \lambda_{1} \lambda_{1}^{*} \}  \Vert  \nabla \Bu\Vert_{L^2(\Omega)}
\| ( e_\Bb, e_\Bu )\|^2  + \epsilon^{-1} C_K (h^{2(k+1)} + h^{2l+2})
\nn \\ 
& \le  \left (\epsilon +C h + \widehat N_2 
 \|  (\Bu, \Bb )\| \right ) 
\| ( e_\Bb, e_\Bu )\|^2  +  \epsilon^{-1} C_K (h^{2(k+1)} + h^{2l+2})
\end{align} 
for some $\epsilon>0$.  Since 
$$ 
a_s(e_\Bu, e_\Bu) + a_m(e_\Bb, e_\Bb)  \ge \min \{ R_e^{-1}, R_m^{-1} \lambda_{0}^{*} \} 
\| (e_\Bu, e_\Bb) \|^2\, , 
$$ 
for $\epsilon$ being small enough, we get 
\begin{align} 
\| ( e_\Bb, e_\Bu) \|  \le C_K (h^{k+1} + h^{l+1})  
\label{3.2-last}
\end{align} 
when $h \le h_0$ for some $h_0>0$. 
The proof of Theorem~\ref{thm_main_result} is complete. 
\quad \endproof 


\subsection{Proof of Corollary \ref{main-coro}} 
Since $\Bu - \Bu_h = \xi_\Bu + e_\Bu$,  
the $L^2$-norm estimate in \refe{main-2} follows \refe{3.2-last} and 
the projection error estimate \refe{s-proj-error}. 
To show the $H^{-1}$-norm estimate in \refe{main-2}, 
we follow the 
approach used for Theorem \ref{thm_Maxwell_projection}. 
By \refe{dual_problem_eqs}, we have 
\begin{align*}
  (\Bb - \Bb_{h}, \Btheta)_{\Omega} 
& =  a_{m} (\Bb - \Bb_{h}, \Bz) + c_{1}(\Bb - \Bb_{h}; \Bu, z) - (\Bb - \Bb_{h}, \nabla \phi)_{\Omega} \\
&  =  a_{m} (\Bb - \Bb_{h}, \Bz - \Bz_{h}) + c_{1}(\Bb - \Bb_{h}; \Bu, \Bz - \Bz_{h}) 
- c_1(\Bb_h, \Bu-\Bu_h, \Bz_h) 
\\ 
& \quad  + (\nabla (r - r_{h}), \Bz - \Bz_{h})_{\Omega}
- (\Bb - \Bb_{h}, \nabla (\phi - \phi_{h}))_{\Omega},  \\
&  =  a_{m} (\Bb - \Bb_{h}, \Bz - \Bz_{h}) + c_{1}(\Bb - \Bb_{h}; \Bu, \Bz - \Bz_{h}) 
+ c_1(\Bb_h, \Bu-\Bu_h, z - \Bz_h) 
\\ 
& \quad  - c_1(\Bb_h, \Bu-\Bu_h,  z) + (\nabla (r - r_{h}), \Bz - \Bz_{h})_{\Omega}
- (\Bb - \Bb_{h}, \nabla (\phi - \phi_{h}))_{\Omega},
\end{align*}
for any $(\Bz_{h}, \phi_{h}) \in \BC_{h}^{k} \times S_{h}^{k}$, where we have used \refe{err_eq1_group1} with $\Bv_h = \Bz_h$. 
By Lemma~\ref{lemma_terms_ineqs}, Lemma~\ref{lemma_discrete_curl_embedding},  Lemma~\ref{thm_mfem_wellposedness} and Theorem \ref{thm_main_result} , 
we further have  
\begin{align*}
 & (\Bb - \Bb_{h}, \Btheta)_{\Omega} \\
& \le C \big( (\Vert \Bu\Vert_{L^{\infty}(\Omega)}+1) \Vert \Bb - \Bb_{h} \Vert_{H(\text{curl}, \Omega)} 
 \Vert \nabla \times (\Bz - \Bz_{h})  \Vert_{L^{2}(\Omega)}   \\ 
& \quad + h^{-1} \Vert \Bb_{h}\Vert_{H(\text{curl}, \Omega)} \Vert \Bu - \Bu_{h} \Vert_{L^2(\Omega)} 
\Vert \nabla \times (z - z_{h})\Vert_{L^{2}(\Omega)} \\
& \quad + \Vert \Bb_{h}\Vert_{L^{6}(\Omega)}\Vert \Bu - \Bu_{h} \Vert_{L^2(\Omega)} 
\Vert \nabla \times z\Vert_{H^1(\Omega)} \\
& \quad  + \Vert \nabla (r - r_{h}) \Vert_{L^{2}(\Omega)} \Vert \Bz - \Bz_{h}\Vert_{L^{2}(\Omega)}  
+ \Vert \Bb - \Bb_{h}\Vert_{L^{2}(\Omega)} \Vert \nabla (\phi - \phi_{h}) \Vert_{L^{2}(\Omega)}  \big)
\\ 
& \le C_K  (h^{k+1} + h^{l+1}) \| \Btheta \|_{H^1(\Omega)} 
\end{align*}
which in turn shows that 
\begin{align}
\label{h-1} 
 \Vert  \Bb - \Bb_{h} \Vert_{H^{-1}(\Omega)} \leq  C_K \left( h^{l+1} + h^{k+1} \right) \, . 
\end{align}
The proof is complete. 
\quad \endproof
\vskip0.1in

\section{Numerical results} 

In this section, we present two numerical examples to confirm our theoretical analysis and show the 
efficiency of methods, one with a smooth solution and one with a non-smooth solution. 
The discrete MHD system \refe{MFEM_mhd} is a system of nonlinear algebraic equations. Iterative algorithms for 
solving such a nonlinear system have been studied by several authors, e.g. s
ee \cite{DongHeZhang, GreifLi2010, Wathen2020, ZhangHeYang} 
for details. Here we use the following Newton iterative algorithm  in our computation: 
\vskip0.1in 

{\it Newton iteration: 
 
For given $(\Bu_h^{n-1}, \Bb^{n-1})$, solve the system 
\begin{align} 
a_s(\Bu_h^n, \Bv_h) 
& + c_0(\Bu_h^{n-1},\Bu_h^{n}, \Bv_h) 
+ c_0(\Bu_h^n, \Bu_h^{n-1}, \Bv_h) - c_1(\Bb_h^{n-1}, \Bv_h, \Bb_h^n) - c_1(\Bb_h^n, \Bv_h, \Bb_h^{n-1}) 
\nn \\ 
& + b_s(p^n_h, \Bv_h) - b_s(q_h, \Bu_h^n) 
\label{new-1}  \\ 
& = (\Bf, \Bv_h) + c_0(\Bu_h^{n-1}, \Bu_h^{n-1}, \Bv_h) - c_1(\Bb^{n-1}, \Bv_h,\Bb^{n-1}), 
\qquad (\Bv_h, q_h) \in \BV_h \times Q_h 
\nn \\ 
a_m( \Bb_{h}^{n}, \Bc_h) & + c_1(\Bb_h^{n-1}, \Bu_h^n, \Bc_h) + c_1(\Bb_h^{n}, \Bu_h^{n+1}, \Bc_h) 
+ b_m(r_h^n, \Bc_h) 
- b_m(s_h, \Bb_h^n) 
\nn \\ 
& = (\Bg, \Bc_h) + c_1(\Bb_h^{n-1} , \Bu_h^{n-1}, \Bc_h), \qquad (\Bc_h, q_h) \in \BC_h \times Q_h 
\label{new-2}
\end{align}  
for $n=1,2,....$, until $\| \nabla (\Bu_{h}^{n}-\Bu_{h}^{n-1}) \|_{L^2(\Omega)} \le 1.0e-10$. 
}
\vskip0.1in 
All computations are performed by using the code FreeFEM++. 
\vskip0.1in 

{\bf Example 4.1.} In the first example, we consider the MHD system \refe{mhd_eqs} 
 on a unit square $(0, 1) \times (0, 1)$ with 
the physics parameters $R_e = R_m = S = 1$. 
We let 
\begin{align} 
& \Bu = \left ( \begin{array}{cc} 
 x^2(x-1)^2 y(y-1)(2y-1) \\ 
 y^2 (y-1)^2 x(x-1)(2x-1) \end{array} \right ) ,
  \qquad 
   p = (2x-1)(2y - 1),
\nn \\ 
&
\Bb = \left ( \begin{array}{cc} 
 \sin(\pi x) \cos(\pi y) 
  \\ 
\sin(\pi y) \cos(\pi x)    
\end{array} 
\right ), \qquad \quad 
r = 0 \, . 
\nn 
\end{align}
be the exact solution of the MHD system and  
choose the source terms $\Bf, \Bg$ and boundary conditions 
correspondingly. 

\begin{figure}[ht]
  \centering
  \begin{tabular}{cc}
    \includegraphics[width = 50mm]{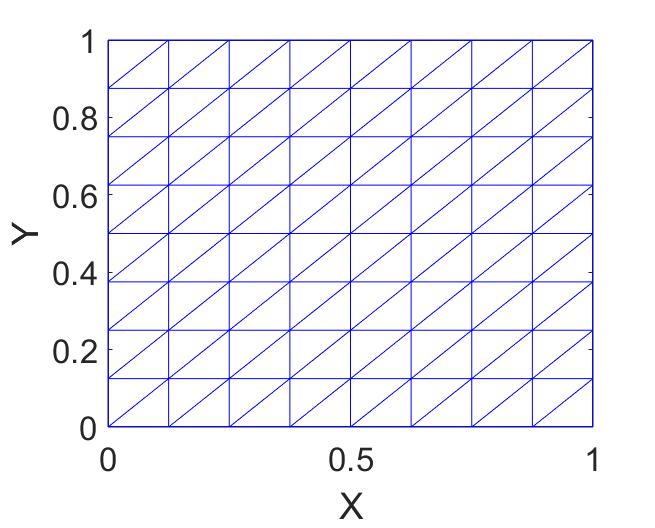}
  \end{tabular}
\vskip-0.2in 
  \caption{A uniform triangular 
  mesh on the unit square.}
  \label{mesh-2d}
\end{figure}

We solve the nonlinear FE system \refe{MFEM_mhd} by the Newton iterative algorithm 
\refe{new-1}-\refe{new-2} with Taylor-Hood/piecewise linear  ($P2-P1$) for $(\Bu, p)$ and 
the lowest-order first type of N\'ed\'elec's edge element and  the lowest-order 
second type of N\'ed\'elec's edge element, respectively, 
for the magnetic field $\Bb$. 
To show the optimal convergence rates, 
a uniform triangular partition with $M+1$ nodes in each direction is used, see Figure 1 for an illustration. 
We present in Table~\ref{table1} numerical results with the lowest-order first type of N\'ed\'elec's edge element 
for $M=4,8,16,32,64,128$. From Table~\ref{table1}, we can observe 
clearly the second-order convergence rate for the velocity $\Bu$ in $H^1$-norm and the pressure in $L^2$-norm 
and the first-order rate for the magnetic field $\Bb$ in $H(curl)$-norm. 
 This confirms  our theoretical analysis, while  in all previous analysis, only the first-order convergence rate 
for the velocity was presented. Our numerical results also show that 
 the lower order approximation to the magnetic field does not 
pollute the accuracy of numerical velocity in $H^1$-norm, although these 
two physical components are coupled strongly in the MHD system. Moreover, we present in Table~\ref{table2} 
numerical results with the lowest-order second type of N\'ed\'elec's edge element  approximation to 
the magnetic field.  The accuracy of the lowest-order second type of N\'ed\'elec's edge element approximation 
is also of the order $O(h)$ in $H(curl)$-norm. Our numerical results show the same convergence rates as numerical 
results obtained by the lowest-order first type of N\'ed\'elec's edge element.

\begin{table}[h]
  \centering
  \begin{center}
      \caption{Errors of Taylor-Hood/lowest-order N\'ed\'elec's edge element of the first type 
for MHD system  
      (Example 4.1). }
      \label{table1}
      \begin{tabular}{c||cc||cc||cc||cc|c}
        \hline 
        \hline 
        $M$  
        &  $\| \nabla (\Bu - \Bu_h) \|_{L^2}$ & Rate & $\| p-p_h \|_{L^2}$ &  Rate & $\| \Bb - \Bb_h \|_{curl} $
        & Rate & $\| r - r_h \|_{H^1}$  \\\hline      
        4  
        & 1.398e-2 &         & 2.774e-2 &         & 8.254e-1 &           & 1.232e-7   \\        
        8  
        & 2.342e-3 & 2.58 & 7.369e-2 & 1.91 & 4.274e-1 & 0.984 & 5.676e-10 \\    
        16
        &4.219e-4 & 2.47 & 1.887e-3 & 1.96 & 2.093e-1 & 0.996 & 2.349e-12 \\ 
        32 
        & 8.983e-5 & 2.23 & 4.750e-4 & 1.99 & 1.047e-1 & 0.999 & 3.732e-13 \\ 
        64
        & 2.130e-5 & 2.08 & 1.190e-4 & 2.00 & 5.237e-2 & 1.00   & 1.553e-12 \\ 
      128
      & 5.250e-6 & 2.02 & 2.976e-5 & 2.00 & 2.168e-2 & 1.00   &  6.273e-12 \\   \hline    
\hline 
    \end{tabular}
  \end{center}
\end{table}

\begin{table}[h]
  \centering
  \begin{center}
      \caption{Errors of Taylor-Hood/lowest-order N\'ed\'elec's edge element of the second type  for MHD system  
      (Example 4.1). }
      \label{table2}
      \begin{tabular}{c||cc||cc||cc||c}
        \hline 
        \hline 
        $M$  &   $\| \nabla (\Bu - \Bu_h)\|_{L^2}$ & Rate & $\| p-p_h \|_{L^2}$ &  Rate & $\| \Bb - \Bb_h \|_{curl} $
        & Rate & $\| r - r_h \|_{H^1}$  \\\hline      
        4  & 1.137e-2 &      & 3.943e-2 &     & 8.093e-1 &   &     2.007e-4   \\        
        8  & 1.829e-3 & 2.63 & 1.041e-2 & 1.92 & 4.095e-1 & 0.982 & 7.128e-6 \\    
        16& 3.669e-4 & 2.32 & 2.640e-3 & 1.98 & 2.054e-1 & 0.996 & 2.331e-7 \\ 
        32 & 8.484e-5 & 2.03 & 6.624e-4 & 1.99 & 1.028e-1 & 0.999 & 7.415e-9 \\ 
        64& 2.075e-5 & 2.03 & 1.658e-4 & 2.00 & 5.140e-2 & 1.00   & 2.334e-10 \\  \hline    
\hline 
    \end{tabular}
  \end{center}
\end{table}

{\bf Example 4.2. } 
The second example is to study numerical solution of  the MHD system 
 on a non-convex $L$-shape domain $\Omega: = (-1,1) \times (-1,1) / 
 (0,1] \times [-1,0)$. The solution of the system may have certain singularity 
 near the re-entrant corner and the regularity of the solution depends upon 
 the interior angles in general. 
  Here we investigate the convergence rates 
 of the method for the problem with a nonsmooth  solution. 
 We set $R_e = R_m = 0.1, S = 1$, and choose the source terms and the boundary conditions such that  the singular solutions are defined by  
\begin{align} 
& \Bu = \left ( \begin{array}{cc} 
 \rho^{\lambda} \left ( 
 (1+\lambda) \sin \theta \phi(\theta) + \cos \theta \phi'(\theta) \right ) \\ 
\rho^{\lambda} \left ( 
- (1+\lambda) \cos \theta \phi(\theta) + \sin \theta \phi'(\theta)  
\right ) 
 \end{array} \right ) , 
  \qquad 
   p = \frac{\rho^{\lambda-1}}{1-\lambda} \left ( 
   (1+\lambda)^2 \phi'(\theta) + \phi'''(\theta) \right ) 
\nn \\ 
&
\Bb = \nabla \left (  \rho^{\frac{2}{3}} \sin \left ( \frac{2\theta}{3} \right ) 
\right ), \qquad \quad 
r = 0 
\nn 
\end{align} 
in the polar coordinate system $(\rho, \theta)$, where
\begin{align} 
\phi(\theta) = 
\sin((1+\lambda) \theta) \frac{\cos(\lambda \omega)}{1+\lambda} - 
\cos((1+\lambda)\theta) - \sin((1-\lambda)\theta) \frac{\cos(\lambda \omega)}{1-\lambda} + \cos((1-\lambda)\theta) 
\nn 
\end{align} 
and the parameters 
$\lambda = 0.54448$ and $\omega=2/3$. 
Clearly $(\Bu, p) \in H^{\lambda+1-\epsilon_0}(\Omega) \times H^{\lambda-\epsilon_0}(\Omega)$ and $\Bb \in 
\BH^{2/3-\epsilon_0}(\Omega)$ for any $\epsilon_0>0$. 
This is a benchmark 
problem in numerical simulations, which was tested  by many authors, 
$e.g.$, see \cite{Badia2013,GreifLi2010, ZhangHeYang}. 

The accuracy of numerical methods usually depends upon the regularity of the exact solution, 
while theoretical analysis given in this paper is based on the assumption of  high regularity. 
Here we use the same method as described in Table~\ref{table1}. 
For the solution of the weak regularity as mentioned above, the 
interpolation error orders on quasi-uniform meshes are 
\begin{align} 
& \| \nabla ( \Bu - \Bu_h ) \|_{L^2(\Omega)} = O(h^{\lambda-\epsilon_0})
\nn \\ 
& \| \Bb - \Bb_h \|_{H(curl, \Omega)} = O(h^{2/3-\epsilon_0}) \, .
\nn 
\end{align} 
To test the convergence rate, a uniform triangulation is made on the 
$L$-shape domain $\Omega$, see Figure 2 (top left) for a sample mesh, where 
$M+1$ nodal points locate in the interval $[0, 1]$.  
We present in Table~\ref{table3} numerical results obtained by the method with  uniform meshes. 
One can see clearly that the orders of numerical approximations for $\Bu$ in $H^1$-norm and for 
$\Bb$ in $H(curl)$-norm are $0.57$ and $0.63$,  respectively, which are very close to the optimal ones in the sense of interpolation. 
It has been noted that a
local refinement may further improve the convergence rate. Here we test 
the method with locally refined meshes, although our analysis was given only for a quasi-uniform mesh. We present three non-uniform meshes in Figure 2 with a finer mesh distribution around the re-entrant corner.  We present in Table 4 numerical results obtained by the method with these four types of meshes in Figure 2.  
From Table~\ref{table4}, we can see the second-order convergence rate 
for the numerical velocity and the first-order convergence rate for the magnetic field approximately. This shows again that the accuracy of the numerical method can be  improved dramatically by using such locally refined meshes.

\begin{figure}[ht]
\vspace{0.1in}
\centering
\begin{tabular}{cc}
\epsfig{file=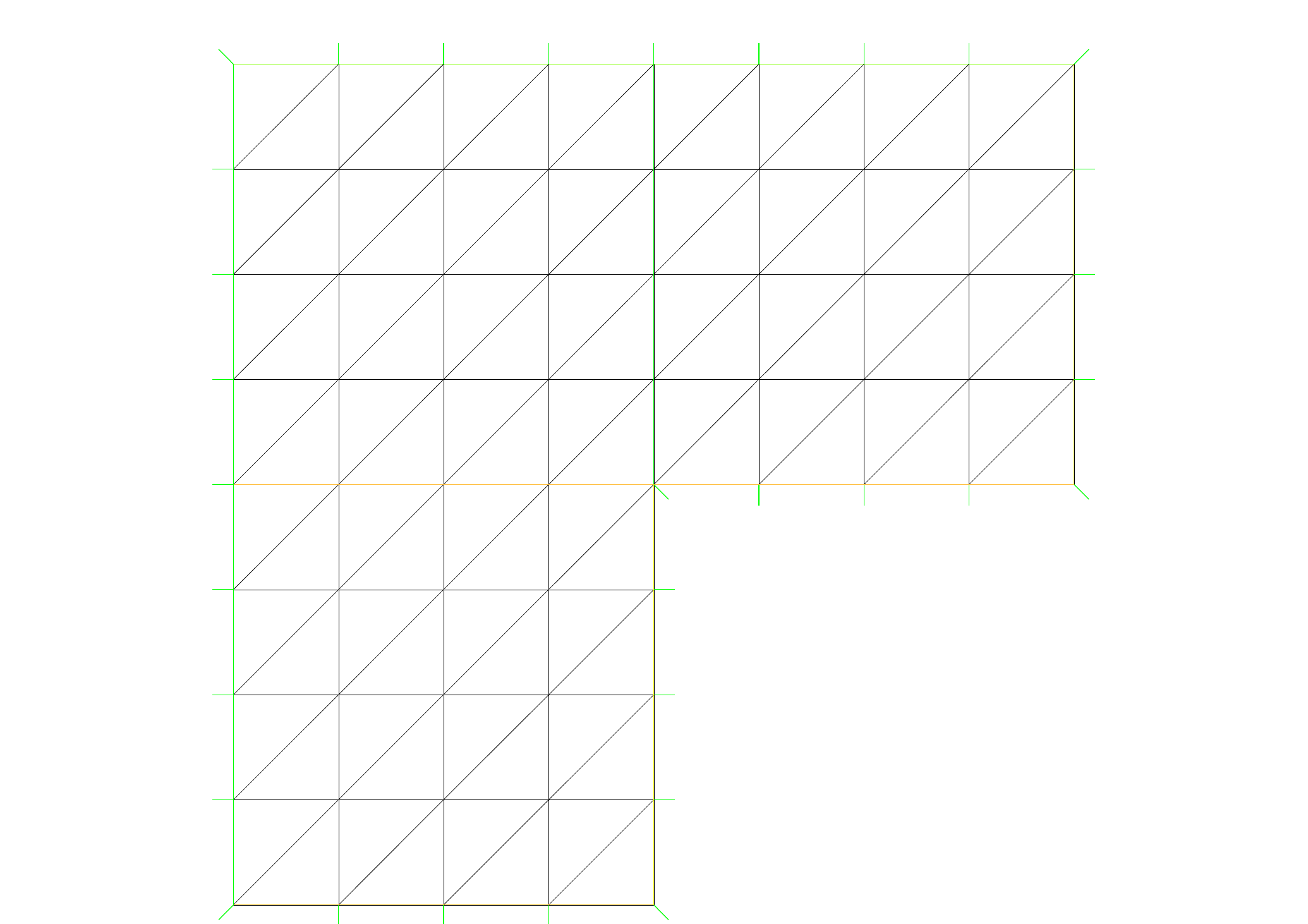,height=1.8in,width=2.5in}
\hspace{0.3in}
&\epsfig{file=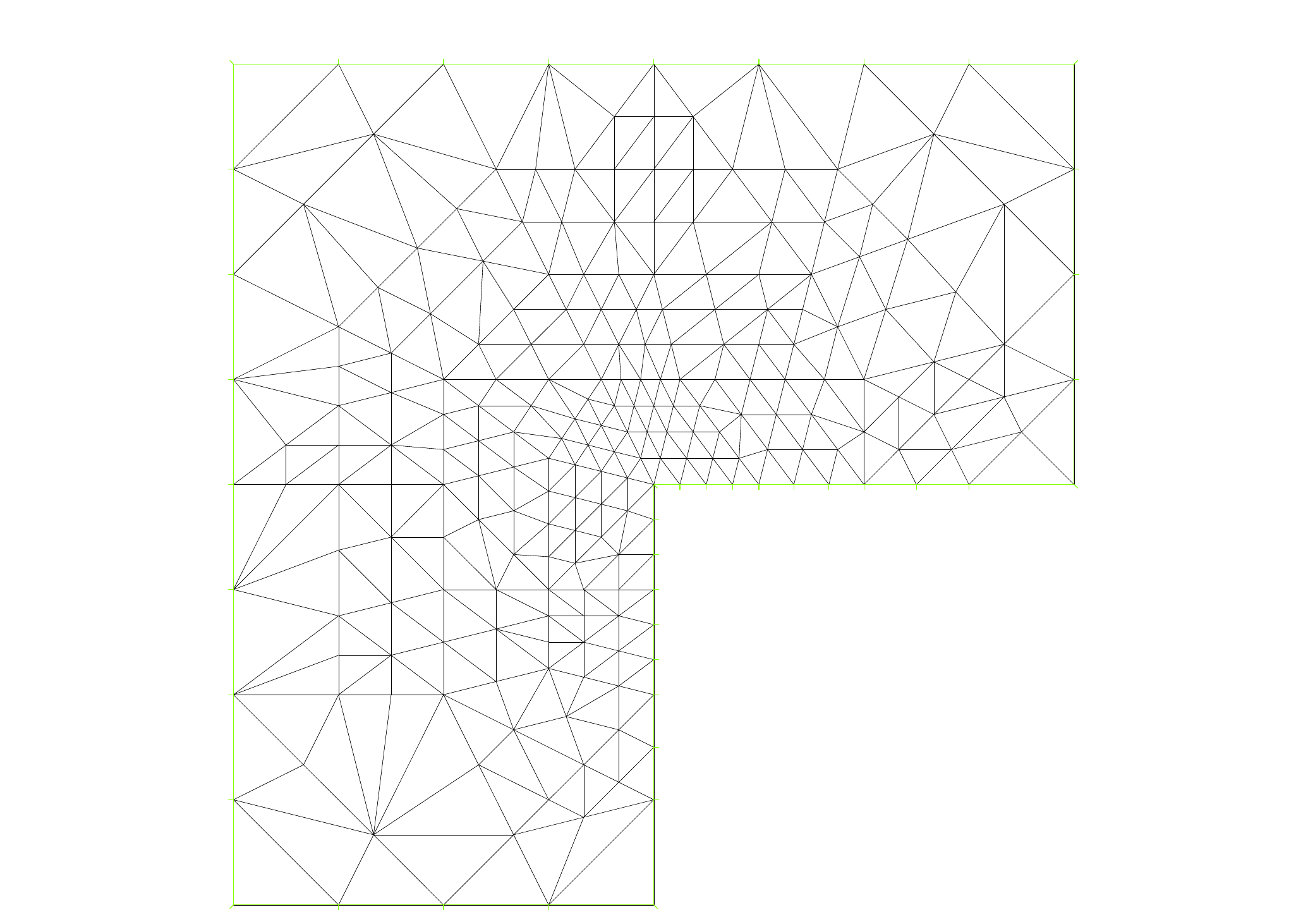,height=1.8in,width=2.5in}
\\ \vspace{0.5in}
\epsfig{file=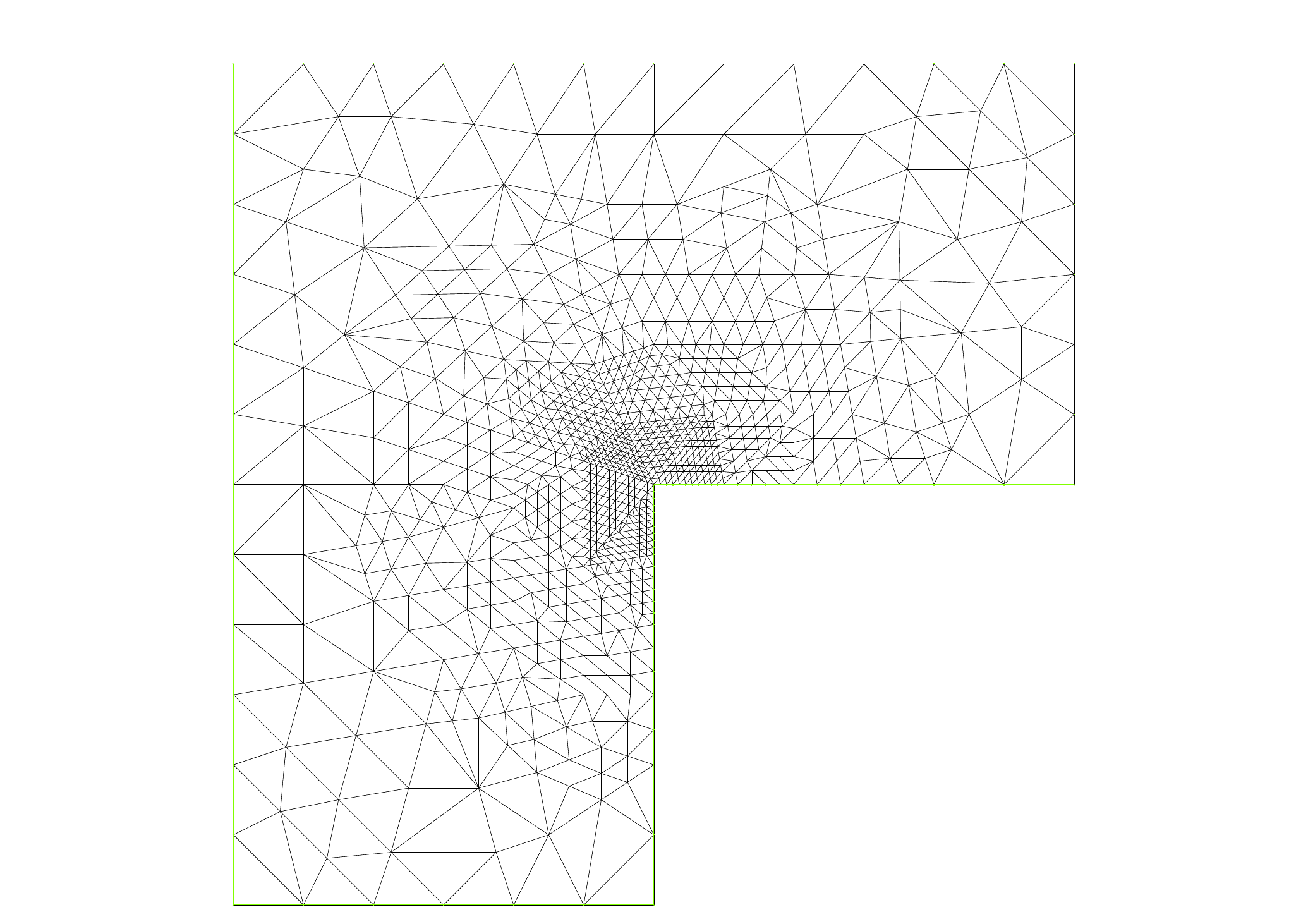,height=1.8in,width=2.5in}
\hspace{0.3in}
&\epsfig{file=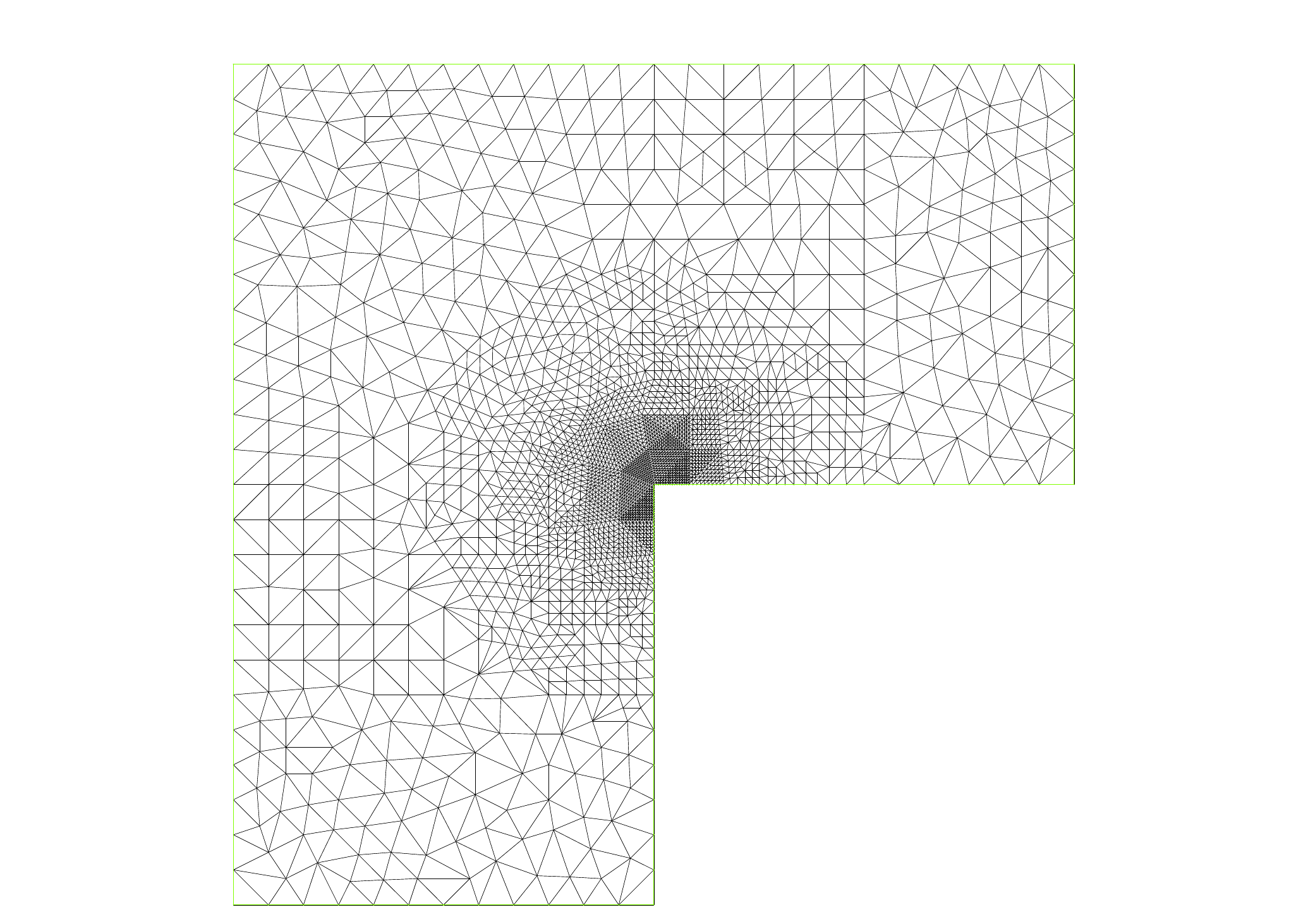,height=1.8in,width=2.5in}
\end{tabular}
\vskip-0.4in
\caption{ Top Left: the first mesh with 65 nodes,  
Top Right: the second mesh with $236$ nodes, 
Bottom Left: the third mesh with $872$ nodes, 
Bottom Right: the fourth  mesh with $2550$ nodes.}
\label{mesh2d_Lshape}
\end{figure}

\begin{table}[h]
  \centering
  \begin{center}
      \caption{Errors of Taylor-Hood/lowest-order Nedelec FEM of the first type 
        for MHD system with the nonsmooth solution in an $L$-shape domain and uniform meshes
      (Example 4.2). }
      \label{table3}
      \begin{tabular}{c||cc||cc||cc||c}
        \hline 
        \hline 
        $M$  & $\| \nabla ( \Bu - \Bu_h)\|_{L^2}$ & Rate & $\| p-p_h \|_{L^2}$ &  Rate & $\| \Bb - \Bb_h \|_{curl} $
        & Rate & $\| r - r_h \|_{H^1}$  \\\hline      
        4  & 1.1281 &              & 4.775 &                & 2.933e-1 &           &  1.470e-3   \\        
        8  & 7.284e-1 & 0.815 & 2.273 & 1.07        & 1.742e-1 & 0.751 & 2.464e-3 \\    
        16& 4.626e-1 & 0.655 & 1.256 & 0.856      & 1.117e-1 & 0.640 & 1.786e-3 \\ 
        32 & 3.216e-1 & 0.525 & 8.141e-1 & 0.814 & 7.279e-2 & 0.618 & 1.052e-3 \\ 
        64& 2.162e-1 & 0.573 & 5.341e-1 & 0.534 & 4.703e-2 & 0.630 & 4.574e-4 \\  \hline    
\hline 
    \end{tabular}
  \end{center}
\end{table}

\begin{table}[h]
  \centering
  \begin{center}
      \caption{Errors of Taylor-Hood/lowest-order Nedelec FEM of the first type
        for MHD system with the nonsmooth solution in an $L$-shape domain and adaptive meshes 
      (Example 4.2). }
      \label{table4}
      \begin{tabular}{c||cc||cc||cc||c}
        \hline 
        \hline 
        Mesh  & $\| \nabla ( \Bu - \Bu_h) \|_{L^2}$ & Rate & $\| p-p_h \|_{L^2}$ &  Rate & $\| \Bb - \Bb_h \|_{curl} $
        & Rate & $\| r - r_h \|_{H^1}$  \\\hline      
        Mesh I    & 1.280e-1 &         & 4.667e-1 &           & 2.912e-1 &         &  1.480e-4   \\        
        Mesh II  & 6.415e-2 & 1.07 & 1.960e-1 & 1.34   & 1.591e-1 & 0.94 & 3.951e-4 \\    
        Mesh III  & 2.017e-2 & 1.77 & 6.236e-2 & 1.75   & 5.562e-2 & 1.60 & 2.771e-5 \\ 
        Mesh IV & 7.521e-3 & 1.85 & 2.243e-2 &  1.91  & 2.715e-2 & 1.33 & 2.668e-5 \\  \hline    
\hline 
    \end{tabular}
  \end{center}
\end{table}

\noindent{\bf Acknowledgments}~The author would like to thank the anonymous referees
for the careful review and valuable suggestions and comments, which have greatly  
improved this article.

\end{document}